\documentclass[11pt,a4paper]{article}
\usepackage[margin=27mm]{geometry}
\usepackage{amsmath,amssymb,amsthm,mathtools}
\usepackage{booktabs,array,longtable}
\usepackage{microtype}
\usepackage[numbers,sort&compress]{natbib}
\usepackage[hidelinks]{hyperref}
\usepackage{url}
\hypersetup{pdftitle={Asymptotic equivalence and exact values for second-order Zarankiewicz numbers},
  pdfauthor={Nikita Lebedev},
  pdfkeywords={Zarankiewicz numbers, rectangle-free matrices, sums of squares, combinatorial constructions}}
\usepackage{enumitem}
\usepackage{needspace}
\setlist{nosep,leftmargin=*}
\newtheorem{theorem}{Theorem}[section]
\newtheorem{lemma}[theorem]{Lemma}
\newtheorem{proposition}[theorem]{Proposition}
\newtheorem{corollary}[theorem]{Corollary}
\theoremstyle{definition}
\newtheorem{definition}[theorem]{Definition}
\newtheorem{remark}[theorem]{Remark}
\newcommand{\zr}{z_{\mathrm{RL}}}
\newcommand{\zs}{z_{\mathrm{SL}}}
\newcommand{\zt}{z_2}
\newcommand{\sos}{\operatorname{sos}}
\newcommand{\BSR}{\operatorname{BSR}}
\newcommand{\RW}{\mathrm{RW3}^{+}}
\newcommand{\SG}{\mathrm{RW3}^{\pm}}

\newcommand{\Q}{\mathbb Q}
\newcommand{\Z}{\mathbb Z}
\newcommand{\fl}[1]{\left\lfloor #1\right\rfloor}

\newcommand{\dd}{\delta}
\title{Asymptotic equivalence and exact values for\\[2pt] second-order Zarankiewicz numbers}
\author{Nikita Lebedev}
\date{October 2026}
\begin{document}
\maketitle
\begin{abstract}
The recursive-line and signed Zarankiewicz numbers maximize the number
of squares in augmentations of a maximum $C_4$-free base, subject to two
sufficient irreducibility criteria. The count includes one square per
base cell and one per selected pair of unused cells. We compare these
parameters with the second-order number, which uses irreducibility itself.
Every maximum $m\times n$ base admits a recursive-line augmentation
with at least $mn/2-C\max(m,n)$ squares, for an absolute constant $C$.
Combining this bound with a two-column extension of known fixed-width
families, we show that all three parameters are asymptotically equivalent,
uniformly as the larger dimension tends to infinity.

For individual displays, a transfer graph shows that once the signed
closure identifies every selected pair, the signed criterion is
equivalent to irreducibility. We determine the second-order number for
every six-column rectangle and give eventual exact formulas for all
three numbers at widths seven, nine and eleven. We also prove signed and recursive-line
equality at $8\times7$ and, together with earlier values, whenever the
shorter side is at most six, except possibly at $14\times4$.
The exact-value results are computer-assisted, using exhaustive
enumeration, checked propositional refutations and symbolic certificates
with a proved lifting argument for arbitrary lengths.
\end{abstract}

\section{Introduction and main results}
An ordinary $C_4$-free matrix can be augmented by pairing unused cells,
each pair representing the square of a binomial in bilinear monomials.
The problem is to maximize the displayed number of squares while ensuring
that no representation by fewer bilinear squares exists. L\"ofberg and Qi
introduced the second-order number $\zt$ and the sufficient recursive
criteria defining $\zr$ and $\zs$ \cite{LQ}. Their hierarchy is
\begin{equation}\label{eq:hierarchy}
 \zr(m,n)\le\zs(m,n)\le\zt(m,n)\le\BSR(m,n).
\end{equation}
The first three parameters require a maximum ordinary singleton base;
$\BSR$ ranges over all SOS biquadratic forms. The author's earlier work
\cite{LebedevAugmented} concerns augmentations under stronger compatibility
conditions. Here we study the signed and recursive-line comparison posed
in \cite{CCQ}, together with its relation to $\zt$.

Three levels must be distinguished: acceptance of an individual display,
the optimum on a prescribed maximum base, and the optimum when the base
may also vary. The criteria can differ at the first two levels even where
the optimized parameters agree. Section~\ref{sec:transfer-graph} explains
the display-level obstruction, and Proposition~\ref{prop:fixed-base-gap}
gives a fixed-base example. Our first result treats every maximum base.
For such a base $F$, write $R(F)$ and $S(F)$ for its largest recursive-line
and signed added-pair counts, with $F$ fixed and the added pairs free.
\begin{theorem}[Uniform half density]\label{thm:uniform-density}
There is an absolute constant $C$ such that every maximum $m\times n$
base $F$ satisfies
\[
 z(m,n)+R(F)\ge\frac{mn}{2}-C\max(m,n).
\]
In particular, uniformly over dimensions and maximum bases as
$\min(m,n)\to\infty$, both $R(F)$ and $S(F)$ equal
$(1/2+o(1))mn$, and
\[
 z_X(m,n)=\bigl(\tfrac12+o(1)\bigr)mn
 \qquad(X\in\{\mathrm{RL},\mathrm{SL},2\}).
\]
For all dimensions, the same constant gives
\[
 S(F)-R(F)\le\frac{z(m,n)}2+C\max(m,n),\qquad
 \zt(m,n)-\zr(m,n)\le\frac{z(m,n)}2+C\max(m,n).
\]
\end{theorem}

The construction organizes long column fans by a matching in a linear set
system. Section~\ref{sec:uniform-density} proves the theorem using that
matching lemma and degree estimates for ordinary maximum bases. Along the
odd-prime incidence subsequence, half density is already implicit in
\cite[Corollary~6.4]{LQ}; the theorem covers every growing rectangle and
every prescribed maximum base.

A separate extension treats fixed widths.
Put
\begin{equation}\label{eq:cap}
 U(m,n)=\fl{\frac{mn+z(m,n)}2},\qquad
 U_n(m)=\fl{\frac{(n+1)m+\binom n2}{2}}.
\end{equation}
The expressions agree in the ordinary terminal range
$m\ge\binom n2$ \cite{Culik}.

\begin{theorem}[Fixed-width deficit]\label{thm:asymptotic}
For every fixed $n\ge3$ and $X\in\{\mathrm{RL},\mathrm{SL},2\}$,
\begin{equation}\label{eq:fixedslope}
 z_X(m,n)=\frac{n+1}{2}m+O_n(1)\qquad(m\longrightarrow\infty).
\end{equation}
More precisely, an explicit construction has
$0\le U_n(m)-\zr(m,n)\le n^3/6+O(n^2)$ for all sufficiently large
$m$ at each width.
\end{theorem}

The cell cap gives the leading upper bound; the extension supplies the
matching lower bound. The constant here depends on $n$. This fixed-width
statement is distinct from Theorem~\ref{thm:uniform-density}, where both
dimensions grow and the estimate is uniform.

\begin{corollary}[Unrestricted asymptotic comparison]\label{cor:unrestricted-asymptotic}
For positive integers $m,n$, uniformly as $\max(m,n)\to\infty$,
\[
 \frac{\zs(m,n)}{\zr(m,n)}\longrightarrow1,
 \qquad \frac{\zt(m,n)}{\zr(m,n)}\longrightarrow1.
\]
\end{corollary}
\begin{proof}
A counterexample sequence either has a subsequence with
$\min(m,n)\to\infty$, contradicting Theorem~\ref{thm:uniform-density},
or has a subsequence with a fixed shorter side. The latter contradicts
Theorem~\ref{thm:asymptotic}, after transposition; widths one and two
have the known exact equalities \cite{LQ}.
\end{proof}

\Needspace{5\baselineskip}

The ratio limit is a leading-order statement, not exact equality. The
extension in Section~\ref{sec:extension} allows arbitrary input displays
and proves the fixed-width theorem without the specialized constructions
used later in the paper.

For individual displays, Theorem~\ref{thm:unary} proves that both criteria
are exact for one added pair on an inclusion-maximal base.
Theorems~\ref{thm:final-graph-relations} and
\ref{thm:orthogonality-completeness} characterize the opposite end of the
closure: once the signed closure identifies all selected pairs, signed acceptance is
equivalent to irreducibility. The identification hypothesis cannot be
dropped.

Our finite and constructive results give exact equality in the following
ranges.
\begin{theorem}[Complete six-column values]\label{thm:six}
For every integer $m\ge2$,
\begin{equation}\label{eq:six}
 \zr(m,6)=\zs(m,6)=\zt(m,6)=
 \begin{cases}
 7,&m=2,\\
 4m,&3\le m\le7,\\
 \fl{(15m+15)/4},&8\le m\le14,\\
 \fl{(7m+15)/2},&m\ge15.
 \end{cases}
\end{equation}
\end{theorem}

At $m=1$ all six cells are fixed, so the common value is $6$.
The new actual upper bounds at $5\times6$, $6\times6$ and $7\times6$
are $20,24,28$. The recursive values and witnesses at the latter two sizes
are already in \cite{CCQ}. At $8\le m\le16$, the constructions of
\cite{CCQ} attain the cell cap. Section~\ref{sec:sixconstruction} supplies
the all-length continuation, starting with their Figure~4 witness.
Theorem~\ref{thm:seven-square} also gives
$\zr(7,7)=\zs(7,7)=\zt(7,7)=33$.

Theorem~\ref{thm:signed-exclusions} proves
\[
 \zr(11,5)=\zs(11,5)=37,\qquad \zr(8,7)=\zs(8,7)=38.
\]
These are signed upper bounds: they leave $\zt(11,5)\in[37,38]$ and
$\zt(8,7)\in[38,39]$. The $11\times5$ equality, combined with the earlier finite values and the
following constructions, implies $\zr=\zs$ whenever the shorter side
is at most six, except possibly at $14\times4$ and its transpose
(Corollary~\ref{cor:five-and-six}).

\begin{theorem}[Odd-width exact formulas]\label{thm:otherwidths}
All three parameters have the following values; the width-five formula
is established in \cite[Theorem~5.4 and Corollary~5.5]{XiChangFive}.
\[
\begin{array}{c|c|c}
 n&\text{row range}&\zr(m,n)=\zs(m,n)=\zt(m,n)\\\hline
 5&m\ge12&3m+5\\
 7&m\ge19&4m+10\\
 9&m\ge36&5m+18\\
 11&m\ge53&6m+27.
\end{array}
\]
The thresholds are least for all three parameters at widths seven, nine
and eleven. At width five, twelve is least for $\zr$ and $\zs$; whether
$\zt(11,5)=38$ remains open.
\end{theorem}
Section~\ref{sec:otherwidths} presents a common one-row construction
for widths seven, nine and eleven.
Chen, L\"ofberg and Qi \cite[Theorem~5.3]{CLQIncidence} determine all
three parameters on the incidence slice $m=\binom n2$ for the orders
listed there, including $\zt(21,7)=94$ and $\zt(36,9)=198$.
These agree with Theorem~\ref{thm:otherwidths}, which determines the
parameters for all $m\ge19$ at width seven and all $m\ge36$ at width nine.
At width eleven it gives $\zr(55,11)=\zs(55,11)=\zt(55,11)=357$,
confirming the $n=11$ value in \cite[Conjecture~4.10]{CLQIncidence}
by a different witness.

The uniform density and structural results have written proofs. The
fixed-width argument starts from the previously established width-three
and width-four families. Our exact-value results are computer-assisted:
the finite exclusions use checked refutations or independent exhaustive
enumerations, and the infinite exact families use symbolic certificate
tables together with the lifting argument in
Lemma~\ref{lem:certificate-lifting}. Appendix~\ref{app:data} identifies
the data and verification methods.
\section{Configurations, coefficient rules and ordinary bounds}
\label{sec:prelim}

Write $u_{rc}=x_ry_c$ for a cell of an $m\times n$ rectangle.
A configuration consists of a singleton set $F$ and a matching $D$ of
unordered pairs of distinct cells outside $F$. All selected supports are
disjoint. Its displayed polynomial and length are
\begin{equation}\label{eq:display}
 P_{F,D}=\sum_{a\in F}u_a^2+
            \sum_{\{a,b\}\in D}(u_a+u_b)^2,
 \qquad N=|F|+|D|.
\end{equation}
We call the display irreducible when its minimum real bilinear SOS
length $\sos(P_{F,D})$ equals $N$. A \emph{limited} configuration has a
$C_4$-free singleton set of size $z(m,n)$. The definitions of $\zt$,
$\zr$ and $\zs$ are those of \cite[Definitions~3.1, 4.11, 7.3]{LQ}.
In particular, $\zr$ is exactly the parameter of
\cite[Definition~4.11]{LQ}: its closure $\RW$ already includes the
complementary-pair rule of \cite[Definition~4.5]{LQ}. Hole grounds
are the unoccupied-companion case of the rectangle-transfer rule in
\cite[Definition~4.1]{LQ}, also used in \cite[Lemma~4.8]{LQ}.
Transport along identified classes and simultaneous complementary
identifications are rules of the same closure. Throughout, a maximum base
has maximum cardinality $z(m,n)$; it is therefore also inclusion-maximal.

In any Gram representation of \eqref{eq:display}, let $v_a$ be the
coefficient vector at cell $a$. An occupied cell has $\|v_a\|=1$, and a
hole has $v_a=0$. Put $\dd_{ab}=1$ when $\{a,b\}\in D$ and zero
otherwise. Distinct cells on one row or column satisfy
$\langle v_a,v_b\rangle=\dd_{ab}$. For the two diagonals $\{a,b\}$ and
$\{c,d\}$ of a genuine rectangle,
\begin{equation}\label{eq:rectangle}
 \langle v_a,v_b\rangle+\langle v_c,v_d\rangle
       =\dd_{ab}+\dd_{cd}.
\end{equation}
A deviation means $\langle v_a,v_b\rangle-\dd_{ab}$.
Accordingly, zero deviations on either diagonal transfer to the other.
If a rectangle has a hole, its diagonal through that hole has zero
inner product and zero prescription, so both deviations vanish. This
hole ground follows directly from the coefficient equations.
If $\dd_{ab}=1$, its zero deviation identifies the two unit vectors.
Two selected complementary diagonals are both identified: their
deviations sum to zero and each is nonpositive by Cauchy--Schwarz.
These are the line, rectangle, saturation and complementary-pair rules
of $\RW$. A certificate is complete if all selected pairs are identified
and the $N$ distinct selected classes are pairwise orthogonal. It then
forces rank at least $N$ in every Gram representation.

The signed rule $\SG$ additionally considers the graph of class-pair
deviations. Rectangle edges reverse signs. A component is forced to zero
if it has a known zero, an odd sign cycle, or unresolved selected-pair
inequalities in both parity classes. One then contracts newly identified
pairs and repeats. We use its established soundness, not a converse.
In particular, failure of either criterion is not itself reducibility.

\begin{lemma}[Ordinary terminal bases and the cell cap]\label{lem:cap}
For $m\ge\binom n2$,
\begin{equation}\label{eq:ordinaryterminal}
 z(m,n)=m+\binom n2.
\end{equation}
Every extremal base in this range has one degree-two row for each column
pair and otherwise degree-one rows. If $m<\binom n2$, then
$z(m,n)\le m+\binom n2-1$. For every limited simple display,
$N\le U(m,n)$.
\end{lemma}
\begin{proof}
If $d_r$ are the singleton row degrees, then
\[
 |F|-m=\sum_r(d_r-1)\le\sum_r\binom{d_r}{2}\le\binom n2.
\]
Equality requires each $d_r$ to be one or two and every column pair to
occur once. The indicated incidence rows and additional singleton rows
attain the bound. If $m<\binom n2$, equality in the same upper
bound would require $\binom n2$ degree-two rows, which is impossible.
This is the classical terminal formula
\cite{Culik}; the equality description is also recorded in
\cite[Lemma~4.6]{XY}. If $H$ is the number of holes, simplicity gives
\begin{equation}\label{eq:cellidentity}
 |F|+2|D|+H=mn,
 \qquad N=\frac{mn+|F|-H}{2}.
\end{equation}
This proves the cap, as in \cite[Proposition~3.2]{LQ}.
\end{proof}

\begin{lemma}[Ordinary six-column values]\label{lem:ordinarysix}
For $6\le m\le14$, the ordinary value is
$z(m,6)=\fl{(3m+15)/2}$.
\end{lemma}
\begin{proof}
Pair counting and $\binom d2\ge2d-3$ give the upper bound.
For the lower bound put $t=\fl{(15-m)/2}$ and take the first $t$
triples in $012,034,135,245$ as singleton row supports, with
zero-based column indices. Their column pairs are disjoint. Add
$m-t$ unused column-pair rows; there are enough since
$m+2t\le15$. This $C_4$-free support has
$3t+2(m-t)=\fl{(3m+15)/2}$ cells.
\end{proof}

\begin{lemma}[Complete-square heredity]\label{lem:heredity}
Every subfamily of the displayed squares of an irreducible SOS is
irreducible. Deleting complete squares (singletons or pairs) also preserves $\RW$
admissibility on the retained coefficient system. Deleting selected pairs,
whose cells become holes, preserves $\SG$ admissibility.
\end{lemma}
\begin{proof}
A shorter SOS for the subfamily, together with the omitted squares,
would shorten the full display. For the recursive assertion, replay a
retained derivation in order. A removed companion cell becomes a hole;
otherwise its rectangle prescription is unchanged. Identifications stay
within complete selected supports. If one of two complementary selected
pairs is removed, the other acquires a hole ground.

For $\SG$, let $D'\subseteq D$ be the retained pairs. We show by induction
over the rounds of the closure of $D$ that its identifications of retained
pairs, and its forced vertices between retained classes, are also obtained
for $D'$. So assume this holds before some round.

Then the classes of $D'$ on retained cells are at least as coarse as those
of $D$. Map each vertex of $D$ between retained classes to the vertex of
$D'$ containing it, or to nothing if its two classes have merged in $D'$;
delete vertices involving a removed class. Under this map:
\begin{itemize}
\item an edge of $D$ goes to an edge of $D'$, or to a ground of $D'$. The
      latter happens when the other diagonal of its rectangle contains a
      removed cell, now a hole, or lies in a class merged in $D'$;
\item line grounds, hole grounds and identified-diagonal grounds persist.
\end{itemize}

Now take a forced component $C$ of $D$. It breaks into pieces, each
either the whole of $C$ or joined within $C$ to a deleted or merged
vertex. A piece of the second kind is grounded in $D'$. The whole of $C$
keeps its ground, its odd closed walk, or its two opposite selected
inequalities, since those pairs are retained and unresolved in $D'$. So
every forced vertex of $D$ between retained classes is forced in $D'$.
If opposite-parity vertices acquire the same image, their connecting
path becomes an odd closed walk. Thus each retained image lies in a
forced component of the new graph, and the induction proceeds.

At the terminal state of $D$, every retained pair is identified and every
two retained classes are certified orthogonal, so $D'$ satisfies $\SG$.
\end{proof}

When restricting rows or columns, we delete every complete square that
meets a deleted cell. A surviving half is not converted into a new
singleton. A numerical limited bound on the retained rectangle is used
only after verifying that its singleton base remains maximum.

\section{Uniform half density}\label{sec:uniform-density}

We first give a construction that retains an arbitrary $C_4$-free base.
The only interaction condition concerns a matching of columns: no fixed
row may contain two of its edges. A rainbow matching will supply this
condition in the asymptotic argument.

\begin{lemma}[Column-fan gluing]\label{lem:column-fans}
Let $J$ be a matching of columns of a $C_4$-free base $F$, such that each
row of $F$ contains at most one edge of $J$. For $ab\in J$, let $H_{ab}$
be the rows missing both $a,b$, and suppose $|H_{ab}|\ge2$.
There is at most one row $z$ fixed in both columns. If it exists, color
each row of $H_{ab}$ by its common fixed column with $z$, giving separate
colors to rows with no such column. Suppose $H_{ab}$ can be ordered
without equal consecutive colors. If no $z$ exists, use any ordering.
Then
\[
 R(F)\ge\sum_{ab\in J}(|H_{ab}|-1).
\]
Here $R(F)$ denotes the pair optimum on this base even when $F$ is not maximum.
\end{lemma}
\begin{proof}
Write an ordering as $q_0,\ldots,q_L$ and select
$A_t=\{(q_t,a),(q_{t+1},b)\}$ for $0\le t<L$.
The two endpoint cells $(q_0,b),(q_L,a)$ are holes. Different fans use
disjoint column sets. First normalize every chain backwards: the terminal
$a$-hole starts the induction. Once $A_{t+1}$ is identified, its representative
$(q_{t+2},b)$ shares column $b$ with the other companion $(q_t,b)$
of $A_t$. This is a column ground, or a hole ground if that companion
is absent. No product proved later is used.

For a foreign column, let $Y_k$ denote its coefficient at $q_k$.
A row grounds $A_tY_k$ when $k=t,t+1$. When $k<t$, an $a$-rectangle
transfers the target to $A_kY_t$. When $k>t+1$, a $b$-rectangle gives
$A_{k-1}Y_{t+1}$; this is a row ground for $k=t+2$, and otherwise an
$a$-rectangle gives $A_{t+1}Y_{k-1}$. The gap decreases by two.
Thus every pair is orthogonal to every foreign coefficient on its path.

Pairs of one fan share a column. For a pair $P$ from another fan, the
previous paragraph applies if one of its rows lies on this path. Otherwise
both rows of $P$ lie outside it, where columns $a,b$ contain only fixed
cells or holes. $C_4$-freeness supplies a hole in that two-by-two block;
choosing the corresponding representatives grounds the product. All
pair/pair products are therefore proved.

Consider a fixed cell $(y,d)$. Path rows and own fan columns are already
covered. If an exterior row $y$ misses $a$ or $b$, that cell is still a
genuine hole: other fans use other columns. It directly grounds the target.
Only the common fixed row $z$ remains. For a fixed $(z,d)$ outside $a,b$,
the ordering gives an endpoint row $r$ of the pair with $(r,d)\notin F$.
Use the rectangle through this endpoint and $(z,d)$. Its companion at
$(r,d)$ is a hole or a pair $Q$ in another fan. In the latter case $z$
misses that fan's other column, since $z$ contains at most one edge of $J$.
Since $z$ is fixed in $d$, it lies outside that fan's path. Its missing
cell in the other column is therefore a genuine hole. Use $Q$'s
representative in that column---available because all fans were normalized
first---to ground its product with $(z,a)$ or $(z,b)$. This is a direct
hole ground, so the common-root cases introduce no circular dependency.

Finally, a rectangle between two fixed cells has a nonfixed companion
by $C_4$-freeness. Its deviation is zero by a hole or a pair/fixed or pair/pair product
already proved. The two companions cannot be one selected pair: that
pair's opposite corners lie on rows missing both of its columns and
therefore cannot be the original two fixed cells. Line products are
immediate. These are all the recursive-line rules, and every singleton
has been retained.
\end{proof}

The coloring in the lemma is well-defined by $C_4$-freeness. A multiset
can be ordered without equal consecutive colors whenever its largest
color class has size at most half the total, rounded up: list the colors
in decreasing multiplicity and fill alternate positions first, then the
remaining positions. Thus the lemma
applies if every column incident with the common row $z$ has degree
$o(|H_{ab}|)$, uniformly over the selected edges.

\subsection{Choosing column pairs}

We use the following matching lemma.
Its proof uses Tutte's theorem and the de Bruijn--Erd\H{o}s incidence theorem.
\Needspace{7\baselineskip}
\begin{lemma}[Packing blocks in a linear set system]\label{lem:block-capacity}
Let a set of $N$ points be partitioned into parts $\mathcal P$, and let
$\mathcal L$ be a family of subsets of size at least two. Distinct members
of $\mathcal L$ meet in at most one point, and each meets every part in
at most one point. For $q\ge1$, there are $q$ disjoint pairs, none inside a part and
at most one inside each member of $\mathcal L$, if and only if
\[
 N\ge2q,\qquad |P|\le N-q\ (P\in\mathcal P),\qquad
 |L|\le N-q+1\ (L\in\mathcal L).
\]
\end{lemma}
\begin{proof}
Necessity follows by counting points outside a part or a member of
$\mathcal L$. First suppose $N=2q$. For each $L$, add two adjacent
vertices and join both to every point of $L$. Join two original points
directly precisely when they lie in different parts and in no common
member of $\mathcal L$. A perfect matching of this graph gives the
required pairs: the two vertices of $L$ either match each other or
match two points of $L$. The intersection assumptions make these
assignments unambiguous.

Apply Tutte's theorem \cite{Tutte}. Delete a vertex set $S$, containing
$h$ original points, both added vertices of $r$ members of $\mathcal L$,
and one added vertex of $s$ others. Thus $|S|=h+2r+s$. Components with
no original point contribute at most $s$ odd components. If the number
$c$ of components containing points is at most $h+2r$, Tutte's inequality
follows. Assume $c>h+2r$.

If $c=1$, then $h=r=0$. For $s>0$, every remaining added vertex still
has a point neighbor, so there is at most one odd component, at most
$|S|$. For $s=0$, the graph has even order and its only point-containing
component is even. If $r=0$ and $c\ge2$, any two points in different
components must belong to the same part: a common member of
$\mathcal L$ would still have an added vertex joining them. All remaining points therefore
belong to one part, giving $c\le2q-h\le q\le h$, a contradiction.

It remains that $r\ge1$ and $c\ge3$. Choose one point from each component.
Every pair of representatives belongs to exactly one part or one of
the $r$ removed members of $\mathcal L$. If the representatives are not
all in one such set, the de Bruijn--Erd\H{o}s incidence theorem
\cite{deBruijnErdos} requires at least $c$ sets. There are at most
$r+\lfloor c/2\rfloor$, since the parts are disjoint. This contradicts
$c>h+2r$. Thus every choice of representatives is contained in one set.
Keeping two representatives and varying a third shows that all remaining
points lie in that same set. A part gives $c\le h$ as above. A member of
$\mathcal L$ gives $2q-h\le q+1$, hence
$c\le2q-h\le h+2\le h+2r$, again a contradiction. Tutte's theorem applies.

For general $N\ge2q$, add $N-2q$ dummy points as one new part, meeting
no member of $\mathcal L$. There are $2(N-q)$ points, and the balanced
case applies with $q$ replaced by $N-q$. Delete the pairs meeting
dummy points. The remaining $q$ pairs have the required properties.
\end{proof}

\subsection{Below a fixed quadratic scale}

\begin{proof}[The first range in Theorem~\ref{thm:uniform-density}]
Orient $m\ge n$, put $M=|F|=z(m,n)$, and fix
$\varepsilon=1/100$, $c_0=\varepsilon^2$.
For all sufficiently large $m$, uniformly over $1\le n\le m$,
\begin{equation}\label{eq:ordinary-lower-uniform}
 M\ge(1-\varepsilon)n\sqrt m.
\end{equation}
Indeed choose a prime $t=(1-o(1))\sqrt m$ with $t^2+t+1\le m$,
and keep $\min(n,t^2+t+1)$ columns of its incidence matrix.
The ratio of this support's size to $n\sqrt m$ tends to one uniformly:
if the plane has fewer than $n$ columns, use $n\le m$.
The prime number theorem and incidence construction suffice
\cite{Singer,Hadamard}.

Suppose $n\le m\le c_0n^2$ and $n$ is sufficiently large.
Pair counting gives $M\le n\sqrt m+m\le2\varepsilon n^2$.
For row degrees $d_r$ and $\mu=M/m$, using
\eqref{eq:ordinary-lower-uniform},
\[
 V:=\sum_r(d_r-\mu)^2
 \le n(n-1)+M-\frac{M^2}{m}\le4\varepsilon n^2,
 \qquad \mu\le\sqrt n+1\le2\varepsilon n.
\]
Consequently $\max_r d_r\le(2\varepsilon+2\sqrt\varepsilon)n<n/2$.
Apply Lemma~\ref{lem:block-capacity} to the $n$ columns, with singleton
parts and with fixed-row neighborhoods of size at least two as
$\mathcal L$. These sets are distinct and meet pairwise in at most one
column, by $C_4$-freeness. With $q=\lfloor n/2\rfloor$, all capacity
conditions hold, including $d_r\le n-q+1$.
We obtain $q$ disjoint column pairs with at most one inside each fixed row.

A column of degree $E$ has at most $E+n-1$ fixed cells in its incident
rows. Applying pair counting to the remaining rows gives
$M\le m+n-1+n\sqrt{m-E}$. Therefore
\[
 \sqrt{1-E/m}\ge1-\varepsilon-\frac{\sqrt m}{n}
                         -\frac1{\sqrt m}\ge1-3\varepsilon,
 \qquad E\le6\varepsilon m<m/5.
\]
For every matched pair $ab$, its common-hole set has at least $3m/5$
rows, and every ordering color class has at most $m/5$ rows (including
fresh singleton classes for large $m$). Lemma~\ref{lem:column-fans} applies.
Writing $e_a$ for a column degree, disjointness of the matched columns gives
\begin{equation}\label{eq:fan-exact-sum}
 M+R(F)\ge M+\sum_{ab\in J}(|H_{ab}|-1)
          \ge |J|(m-1).
\end{equation}
Here $|H_{ab}|=m-e_a-e_b+|N_F(a)\cap N_F(b)|$, and the sum of all
matched column degrees is at most $M$. Thus
$M+R(F)\ge\lfloor n/2\rfloor(m-1)\ge mn/2-m$.
\end{proof}

\subsection{The quadratic and larger ranges}

The following elementary replacement bounds row degrees when the
number of rows is at least a fixed multiple of $n^2$.

\begin{lemma}\label{lem:bounded-row-trade}
For each $c>0$, the row degrees of every maximum $m\times n$ base
with $m\ge cn^2$ are bounded by a constant depending only on $c$,
for all sufficiently large $n$.
\end{lemma}
\begin{proof}
Pair counting gives $M/m\le1+1/\sqrt c$. Choose an integer
$K\ge2(1+1/\sqrt c)$. At least $m/2\ge n$ rows have degree at most
$K$, for sufficiently large $n$; these differ from the high row used below.
Suppose a row has degree $a>2\binom{K+1}{2}^2+1$ and put $k=K+1$.
Remove that row and $a$ low-degree rows. Let $A$ be its $a$ columns;
every surviving row meets $A$ at most once.

There are $a+1$ $k$-subsets of $A$ with no repeated pair. Construct
them greedily: after at most $a$ choices there are at most
$a\binom k2$ forbidden pairs, and a uniformly chosen $k$-subset
contains fewer than one of them in expectation, since
\[
 \frac{a\binom k2^2}{\binom a2}<1.
\]
It therefore avoids every forbidden pair for some choice.
Insert these $a+1$ blocks in the freed rows. The result is $C_4$-free.
At most $a+aK=ak$ cells were removed, but $(a+1)k$ were inserted,
contradicting the maximum cardinality of $F$.
\end{proof}

\begin{proof}[Completion of Theorem~\ref{thm:uniform-density}]
First assume $m\ge cn^2$ for fixed $c>0$ and sufficiently large $n$.
By Lemma~\ref{lem:bounded-row-trade}, choose an integer $D\ge2$,
with $D=O_c(1)$, bounding all row degrees. Color each column pair by its unique common fixed row,
or give it a fresh color when none exists. Set $\varepsilon=1/5$. Call a column heavy if
its degree exceeds $\varepsilon m$. Since $M=O_c(m)$, there are
$O_c(1)$ heavy columns. Omit these vertices from the column graph.
Mark as forbidden each edge whose common fixed row meets a heavy column.
For one heavy column $h$, $C_4$-freeness gives
\[
 \sum_{z\sim h}\binom{d_z-1}{2}
 \le\frac{D-2}{2}\sum_{z\sim h}(d_z-1)
 \le\frac{(D-2)(n-1)}2.
\]
Thus only $O_c(n)$ edges are forbidden.

Omit one further vertex if necessary so the number $N=n-O_c(1)$
of remaining columns is even. This parity deletion does not enlarge the set
of forbidden row colors. Choose a uniform perfect matching of the complete
graph on these vertices, with $N\ge4$. A given edge is selected with
probability $1/(N-1)$, so the expected number of forbidden matching edges
is $O_c(1)$. Two disjoint edges are selected together with probability
$1/((N-1)(N-3))$. Each row color has at most $B=\binom D2$ edges, and
\[
 \sum_{\text{row colors }i}\binom{b_i}{2}
 \le\frac{B-1}{2}\binom n2=O_c(n^2).
\]
Thus the expected number of equal-color pairs of matching edges is $O_c(1)$;
intersecting edges cannot both be selected, and fresh colors never collide.
Some matching has total cost $O_c(1)$. Delete its forbidden edges
and all but one edge of each repeated color. A color appearing $r$ times
costs at most $r-1\le\binom r2$ deletions. We obtain a rainbow matching
of $n/2-O_c(1)$ edges.

For each retained matching edge $ab$, both columns have degree at most
$m/5$, and their common fixed row, if any, meets no heavy column.
Hence $|H_{ab}|\ge3m/5$, while every ordering color class has size at
most $m/5$. Lemma~\ref{lem:column-fans} and the exact summation
\eqref{eq:fan-exact-sum} give
\begin{equation}\label{eq:quadratic-error}
 M+R(F)\ge |J|(m-1)\ge\frac{mn}{2}-O_c(m).
\end{equation}
The constants depend only on $c$, uniformly in $m\ge cn^2$ and in $F$.

Take the fixed $c=c_0=10^{-4}$ from the first range. Together, the two
arguments give $M+R(F)\ge mn/2-Cm$ for some absolute $C$, for all
$m\ge n\ge n_0$. Increasing $C$ to at least $n_0/2$ covers $n<n_0$
as well, since the claimed lower bound is then nonpositive.
Every line, rectangle and normalization rule is invariant under exchanging
rows and columns, so transposition gives the stated bound for all dimensions.

The cell caps $S(F)\le(mn-M)/2$ and $\zt(m,n)\le(mn+M)/2$, together
with $\zr(m,n)\ge M+R(F)$, give both comparison bounds.
Finally $M/(mn)\le1/n+1/\sqrt m\to0$ when both dimensions grow.
The lower bound just proved and the cell cap therefore give the uniform
half-density limits.
\end{proof}

\section{A universal two-column extension}\label{sec:extension}

The exact-width constructions in Sections~\ref{sec:sixconstruction} and
\ref{sec:otherwidths} use special interfaces. The following
extension works for an arbitrary simple input, at the cost of a bounded
number of holes at each fixed width.

\begin{theorem}[Short path extension]\label{thm:extension}
Let $P$ be any simple display on $m\ge1$ rows and $n\ge1$ columns.
Add columns $A,B$ and pairs
\[
 Q_k=\{(k-1,A),(k,B)\},\qquad1\le k\le m-1,
\]
leaving $(0,B),(m-1,A)$ empty. The new polynomial $P^+$ satisfies
$\sos(P^+)=\sos(P)+m-1$. Irreducibility, $\RW$ and $\SG$ each hold
before the extension if and only if they hold afterward.

If one also adds $2n+1$ fresh two-singleton rows, one for each column
pair involving $A$ or $B$, let $f,d,N,H$ count the singletons, pairs,
displayed squares and holes. The exact changes are
\begin{align}\label{eq:extensioncounts}
 m'&=m+2n+1,&n'&=n+2,\notag\\
 f'&=f+4n+2,&d'&=d+m-1,\notag\\
 N'&=N+m+4n+1,&H'&=H+2n^2+n+2.
\end{align}
The minimum SOS length increases by $m+4n+1$. A maximum terminal
singleton base remains maximum, and
$m'-\binom{n+2}{2}=m-\binom n2$.
\end{theorem}

\begin{proof}
For $m\ge2$, the first pair normalizes through $(0,B)$, and the ordinary wave
normalizes every $Q_k$. Distinct new classes share a column and are
orthogonal. Consider any old coefficient $X_{l,c}$, including an
unresolved half. A same-row ground handles $k=l$. Otherwise the
$B,c$ rectangle gives
\[
 \langle Q_k,X_{l,c}\rangle
   \longleftrightarrow \langle X_{k,c},B_l\rangle.
\]
If $l=0$ or $X_{k,c}$ is a hole, it is grounded. Otherwise
$B_l=Q_l$, whose outgoing representative is $(l-1,A)$.
If $k=l-1$, this is a row ground. In every other case a second
rectangle gives
\begin{equation}\label{eq:universaldrift}
 \langle Q_k,X_{l,c}\rangle
   \longleftrightarrow\langle A_k,X_{l-1,c}\rangle
   =\langle Q_{k+1},X_{l-1,c}\rangle,
\end{equation}
unless $k=m-1$, when $A_k$ is a hole. The index move $(k,l)\mapsto
(k+1,l-1)$ reaches a hole or a row ground. No old identification is
assumed. Cases with coincident rows use the line rule instead. Thus the new chain is orthogonal to every old
coefficient in every feasible Gram representation.

The additional singleton rows are sparse. A singleton in a new column
is orthogonal to the chain by a common column. One in an old column
uses the other new column, a hole in that row. Its products with old
coefficients are line/hole grounds; the sole possible exception
transfers to a just-established new-singleton/chain product. The new
two-cell row supports are distinct, so their singleton set is $C_4$-free
and their mutual products are likewise grounded.

After merging each new pair, every feasible Gram matrix is the direct
sum of an old feasible block and an identity block of the asserted
size. Conversely such a direct sum always represents the new
polynomial. This proves the rank increments and irreducibility
equivalence. All the new identities are $\RW$ inferences. An old/old
rectangle stays wholly old, and no new pair gives an old cell another
representative; consequently new closure cannot repair an old
uncertified component, for either criterion. The converses follow.
The case $m=1$ has no chain pairs and is included. Finally the new
singleton rows represent exactly the new column pairs, proving the
maximum-base assertion and the counts.
\end{proof}

\begin{proof}[Proof of \eqref{eq:fixedslope}]
Start at width $b=3$ for odd $n$, using the prior exact family
$\zr(r,3)=2r$ for $r\ge3$ \cite{LQ}; it has three holes.
For even $n$ start at $b=4$, using \cite{CC} for $r\ge15$, with
zero or one hole. Repeated application of Theorem~\ref{thm:extension}
reaches every target row count satisfying
\[
 m\ge\binom n2\quad(n\text{ odd}),\qquad
 m\ge\binom n2+9\quad(n\text{ even}).
\]
The hole count is at most
\begin{equation}\label{eq:holepolynomial}
 h_b+\sum_{w=b,b+2,\ldots,n-2}(2w^2+w+2),
 \qquad h_3=3,\quad h_4=1.
\end{equation}
It is independent of $m$ and equals $n^3/3+O(n^2)$.
Equation~\eqref{eq:cellidentity} bounds the deficit from $U_n(m)$ by
half this quantity, with the appropriate integer rounding. The prior
cap and hierarchy prove \eqref{eq:fixedslope} for all three parameters.
\end{proof}

The fixed-width gaps $\zt-\zr$ and $\zs-\zr$ are consequently bounded
as $m$ grows, but the argument permits a nonzero bounded gap. The
individual-configuration equivalence in Theorem~\ref{thm:extension}
does not convert arbitrary optimized extremizers into one another.

\section{Exactness at the first and final stages}\subsection{One pair on a maximal singleton base}\label{sec:unary}

An inclusion-maximal $C_4$-free base need not be maximum. Nevertheless,
maximality already makes the one-pair sufficient criteria exact.

\begin{theorem}[One-pair completeness]\label{thm:unary}
Let $F$ be inclusion-maximal and $C_4$-free in the full rectangle, and
let $a,b\notin F$ be distinct. The display with singleton base $F$ and
one selected pair $\{a,b\}$ is irreducible if and only if it satisfies
$\RW$, if and only if it satisfies $\SG$.
\end{theorem}

\begin{proof}
Write $a=(r,c)$, $b=(s,d)$. On a $C_4$-free base, the unary recursive
test has the following two geometric conditions:
\begin{enumerate}[label=(\alph*)]
\item If the pair is nondegenerate, its companion cells $(r,d),(s,c)$
      are not both fixed.
\item There is no fixed cell $(t,e)$, with $t\notin\{r,s\}$ and
      $e\notin\{c,d\}$, for which
      $(r,e),(s,e),(t,c),(t,d)$ are all fixed.
\end{enumerate}
Repeated corners in a degenerate pair are counted only once.
Under (a), the pair normalizes by a line or hole. Under (b), every
pair/fixed product has a line or companion-hole ground. The remaining
fixed/fixed products then follow from $C_4$-freeness and
\eqref{eq:rectangle}. Thus (a),(b) imply $\RW$.

Every failure has an actual shortening. In a nondegenerate fork set
\[
 E=x_ry_c+x_sy_d,\quad f=x_ty_e,\quad
 g=x_ry_e,\ h=x_ty_c,\ j=x_sy_e,\ k=x_ty_d.
\]
These six displayed forms satisfy $Ef=gh+jk$, and
\begin{equation}\label{eq:fork6}
 E^2+f^2+g^2+h^2+j^2+k^2
       =(E+f)^2+(g-h)^2+(j-k)^2.
\end{equation}
For a degenerate fork there are five distinct forms with $Ef=g(h+k)$,
and a four-square replacement is
\begin{equation}\label{eq:fork5}
 \left(E+\frac f{\sqrt2}\right)^2+\frac{f^2}{2}
 +\left(g-\frac{h+k}{\sqrt2}\right)^2+\frac{(h-k)^2}{2}.
\end{equation}

Suppose (a) fails. Let $u=(r,d)$, $v=(s,c)$ be fixed. Since $u_au_b=u_uu_v$,
\begin{equation}\label{eq:switch}
 u_u^2+u_v^2+(u_a+u_b)^2
       =u_a^2+u_b^2+(u_u+u_v)^2.
\end{equation}
Put $F'=(F\setminus\{u,v\})\cup\{a,b\}$. If $F'$ contains a $C_4$,
its four singleton squares reduce to two by the rectangle identity.
Otherwise, maximality of $F$ supplies a fixed rectangle completed by
adding $a$. It must contain $u$ or $v$, since otherwise it survives in
$F'$, and cannot contain both, since its fourth corner would be $b$.
It therefore produces, possibly after transposition, a five-square
subsum
\begin{equation}\label{eq:strip5}
 (Au+Bv)^2+(Av)^2+(Bu)^2+(Cu)^2+(Cv)^2.
\end{equation}
Here $A,B,C$ are three row variables and $u,v$ two column variables.
Its explicit four-square replacement is
\begin{align}\label{eq:sqrtthree}
 &\left(Au+\frac{B+\sqrt3C}{2}v\right)^2
 +\left(\frac{\sqrt3B-C}{2}v\right)^2 \notag\\
 &\hspace{9mm}+\left(Av+\frac{B-\sqrt3C}{2}u\right)^2
 +\left(\frac{\sqrt3B+C}{2}u\right)^2.
\end{align}
Expansion gives exactly \eqref{eq:strip5}. Every replacement is bilinear
and shortens a genuine displayed subfamily. Thus failure of (a) or (b)
is reducibility. Soundness $\RW\Rightarrow\SG\Rightarrow$ irreducibility
completes the equivalence.
\end{proof}

The strip obstruction itself is prior: the known two-column bound already
implies that five displayed squares in a $3\times2$ strip are reducible
\cite[Appendix~B.1]{LQ}. Formula~\eqref{eq:sqrtthree} makes the local
certificate explicit. The additional point in Theorem~\ref{thm:unary}
is that maximality forces one of the shortening configurations whenever
unary normalization fails.

\begin{corollary}\label{cor:unaryfilter}
In an irreducible limited configuration, every individual selected pair,
together with the entire singleton base, satisfies both unary tests.
\end{corollary}
\begin{proof}
A maximum base is inclusion-maximal. Apply Lemma~\ref{lem:heredity}
and Theorem~\ref{thm:unary} to each one-pair subfamily.
\end{proof}
This is a necessary preprocessing condition, not joint sufficiency for
several pairs. The maximality hypothesis also cannot be dropped: on
$2\times2$, the three-square display
$(x_1y_1+x_2y_2)^2+(x_1y_2)^2+(x_2y_1)^2$ is irreducible but fails the
signed test on its nonmaximal two-cell singleton base. Its Gram matrices
have blocks $\left(\begin{smallmatrix}1&t\\t&1\end{smallmatrix}\right)$
and $\left(\begin{smallmatrix}1&1-t\\1-t&1\end{smallmatrix}\right)$;
positivity is $0\le t\le1$, with minimum rank three.

\subsection{The final transfer graph}\label{sec:transfer-graph}

This section makes precise how the two criteria relate once every selected
pair has been identified. We show that, from that point on, the orthogonality
phase of $\SG$ is exact: it fails exactly when the displayed forms satisfy a
nontrivial quadratic relation. We use only simplicity: no cell lies in two
displayed squares. The base need not be $C_4$-free or maximum. The rules of
$\RW$ and $\SG$, and their soundness proofs in \cite{LQ}, use only
simplicity, although they are stated there for limited graphs.

\subsubsection{The graph and its relations}

Let $(F,D)$ be a simple display. Its \emph{classes} are the displayed
squares' supports: the singletons of $F$ and the pairs of $D$. For a
class $X$ put $f_X=\sum_{a\in X}u_a$. Thus
\[
 P_{F,D}=\sum_X f_X^2 \qquad\text{and}\qquad N=|F|+|D|.
\]

\begin{definition}\label{def:final-graph}
The \emph{final transfer graph} $\Gamma(D)$ is defined as follows.
\begin{itemize}
\item \emph{Vertices} are the unordered pairs $\{X,Y\}$ of distinct
      classes. A \emph{cross pair} of $\{X,Y\}$ is a pair of cells
      $x\in X$, $y\in Y$.
\item A vertex is \emph{grounded} if some cross pair
      \begin{enumerate}[label=(\roman*)]
      \item lies in one row or one column, or
      \item spans a genuine rectangle whose other diagonal contains a
            hole or is a pair of $D$.
      \end{enumerate}
\item \emph{Edges}: each genuine rectangle with all four corners occupied,
      whose two diagonals are cross pairs of vertices $v$ and $w$, gives an
      edge $vw$. It is a loop when $v=w$.
\end{itemize}
A component is \emph{grounded} if it contains a grounded vertex. It is
\emph{bipartite} if it has no closed walk of odd length.
\end{definition}

Every loop is at a grounded vertex. Suppose the two diagonals
$\{(r,c),(s,d)\}$ and $\{(r,d),(s,c)\}$ of one rectangle give the same
vertex. Then two corners in one row, or in one column, belong to the same
class. Another cross pair of that vertex then lies in one line, so the
vertex is grounded by~(i).

\Needspace{10\baselineskip}
\begin{theorem}\label{thm:final-graph-relations}
Let $C=(C_{XY})$ be a real symmetric matrix indexed by classes. Then
$\sum_{X,Y}C_{XY}f_Xf_Y$ vanishes identically if and only if $C$ has zero
diagonal and the function $h(\{X,Y\})=C_{XY}$ satisfies:
\begin{itemize}
\item $h$ vanishes at grounded vertices;
\item $h(v)=-h(w)$ along every edge $vw$.
\end{itemize}
The space of such relations therefore has dimension equal to the number of
ungrounded bipartite components of $\Gamma(D)$.

If such a component exists, then $\sos(P_{F,D})\le N-1$.
\end{theorem}

\begin{proof}
We compare coefficients of the biquadratic monomials.
\begin{itemize}
\item \emph{Squares $u_a^2$.} This monomial occurs only in $f_X^2$ for the
      class $X$ of $a$, with coefficient $C_{XX}$. Hence $C_{XX}=0$.
\item \emph{Line monomials} $x_r^2y_cy_d$ and $x_rx_sy_c^2$. Each arises
      from a single cell pair, and when its cells lie in distinct classes
      the coefficient is $2h(v)$. This is condition~(i).
\item \emph{Rectangle monomials} $x_rx_sy_cy_d$. Each arises exactly from
      the two diagonals of one rectangle. A diagonal contributes $2h$ of its
      vertex when both cells are occupied and in distinct classes, and
      nothing otherwise (a hole, or a single class, where $C_{XX}=0$). So
      the coefficient vanishes exactly when $h(v)+h(w)=0$ along an edge,
      or $h(v)=0$ under ground~(ii).
\end{itemize}
On a component, these conditions force $h\equiv0$ unless the component is
ungrounded and bipartite. On such a component the solutions are the
multiples of its $\pm1$ colouring.

Now take such a colouring $H$. It is symmetric, nonzero, has zero diagonal
and trace, and satisfies $f^{\mathsf T}Hf\equiv0$. With
$\lambda=\lambda_{\max}(H)>0$, the matrix $W=I-H/\lambda$ is positive
semidefinite of rank at most $N-1$, and $P_{F,D}=f^{\mathsf T}Wf$.
Factoring $W$ gives at most $N-1$ bilinear squares; this is the rank step
of \cite[Lemma~2.13]{XY}.
\end{proof}

\subsubsection{Completeness after identification}

In any SOS representation of $P_{F,D}$, the coefficient equations of
Section~\ref{sec:prelim} hold for the cell vectors. Every rule of $\RW$
and $\SG$ is a deduction from these equations, together with the
Cauchy--Schwarz inequality for the unit vectors of a selected pair. We say that a criterion
\emph{identifies} $D$ if its closure resolves every selected pair.

\begin{theorem}\label{thm:orthogonality-completeness}
Let $(F,D)$ be simple.
\begin{enumerate}[label=(\alph*)]
\item If $(F,D)$ satisfies $\SG$, then $\Gamma(D)$ has no ungrounded
      bipartite component.
\item If $\SG$ identifies $D$ and $\Gamma(D)$ has no ungrounded bipartite
      component, then $(F,D)$ satisfies $\SG$.
\item If $(F,D)$ satisfies $\RW$, then every component of $\Gamma(D)$ is
      grounded. Conversely, if $\RW$ identifies $D$ and every component is
      grounded, then $(F,D)$ satisfies $\RW$.
\end{enumerate}
In particular, if $\SG$ identifies $D$, the following are equivalent:
\begin{itemize}
\item $(F,D)$ satisfies $\SG$;
\item $\Gamma(D)$ has no ungrounded bipartite component;
\item $P_{F,D}$ is irreducible.
\end{itemize}
\end{theorem}

\begin{proof}
\emph{(a).} Let $K$ be an ungrounded bipartite component with colouring
$H$, and take $0<t\le1/\lambda_{\max}(H)$. Then $I-tH$ is positive
semidefinite with unit diagonal; write $I-tH=B^{\mathsf T}B$ with columns
$w_X$. Put $v_a=w_X$ for each $a\in X$, and $v_h=0$ at holes. Since
$f^{\mathsf T}Hf\equiv0$, these vectors come from the SOS representation
\[
 P_{F,D}=\sum_i\bigl(\textstyle\sum_a(v_a)_iu_a\bigr)^2 .
\]
For $x\in X$ and $y\in Y$ with $\{X,Y\}\in K$, we get
$\langle v_x,v_y\rangle=-tH_{XY}\ne0$. But an accepting closure
certifies $\langle v_x,v_y\rangle=0$ in every representation.

\emph{(b).} Once every pair is resolved, the current classes are the
displayed supports. The signed transfer graph of that stage then contains
all grounds and edges of $\Gamma(D)$:
\begin{itemize}
\item the line rule gives ground~(i);
\item transfer from a hole, or from an identified pair's diagonal, gives
      ground~(ii);
\item the rectangle identities give the sign-reversing edges, since all
      prescriptions between distinct final classes vanish.
\end{itemize}
Grounded components are forced by propagation, and every other component
has an odd cycle. So every pair of classes is certified orthogonal.

\emph{(c).} Every certified orthogonality $p\perp q$, with $p,q$ in
distinct final classes $X,Y$, lies in a grounded component of $\Gamma(D)$.
We prove this by induction over the $\RW$ derivation.
\begin{itemize}
\item \emph{Line rule.} This gives ground~(i).
\item \emph{Saturation.} This stays within the same vertex, because
      identifications merge only the two cells of a selected pair.
\item \emph{Complementary-pair rule.} This only identifies pairs and
      produces no orthogonality.
\item \emph{Rectangle transfer} from the other diagonal $\{r,s\}$.
      \begin{itemize}
      \item If that diagonal contains a hole, or is an identified pair, we
            get ground~(ii).
      \item If $r\perp s$ was certified, then $\delta(\{r,s\})=0$, so
            $\{r,s\}$ is not a selected pair and $r$, $s$ lie in distinct
            classes. The rectangle is an edge to their vertex, which lies in
            a grounded component by induction.
      \end{itemize}
\end{itemize}
Acceptance certifies every vertex. The converse is as in~(b), without
odd cycles. Saturation moves a certificate between the cross pairs of one
vertex, and rectangle transfer carries it along an edge.

For the final equivalence: (a) and (b) give the first equivalence;
Theorem~\ref{thm:final-graph-relations} and the soundness of $\SG$
\cite[Theorem~7.2]{LQ} give the second.
\end{proof}

\begin{remark}\label{rem:final-graph}
\emph{The identification hypothesis is necessary.} Consider the
$4\times4$ display $T$ of \cite[(40)]{XY}: the four diagonal cells as
singletons, plus the six pairs $F^+$ of \cite[(38)]{XY}.
\begin{itemize}
\item A direct check shows that all $15$ components of $\Gamma$ are
      grounded. By Theorem~\ref{thm:final-graph-relations} there is no
      product relation, in agreement with \cite[Proposition~5.4]{XY}.
\item The display has no holes, no degenerate pairs and no complementary
      pairs, so neither closure identifies any pair.
\item Yet $T$ is a sum of nine squares \cite[Theorem~5.3]{XY}.
\end{itemize}
As Xu and Yan remark after their Proposition~5.4, shorter representations
may use forms outside the displayed span. Identification is what excludes
them here.

\emph{Related work.} The proof of \cite[Proposition~5.4]{XY} already
observes that relations have zero diagonal and correspond to the kernel of
the product map for one display. \cite[Remark~7.8]{LQ} describes the Gram
perturbation mechanism. In their formulation,
\cite[Proposition~2.3]{CLQIncidence} shows that after an identification
prefix, grounding of every component suffices for $\RW$; this is the
converse half of Theorem~\ref{thm:orthogonality-completeness}(c).
Here we add the necessity direction, the description of all product
relations (Theorem~\ref{thm:final-graph-relations}), and the equivalence
of $\SG$ with irreducibility when the signed closure identifies $D$.

\emph{Relation to \cite[Remark~7.8]{LQ}.} That remark declines a converse
for ungrounded bipartite signed components. The remark remains appropriate
whenever identification is incomplete, including at fixed points with
unresolved pairs. Theorem~\ref{thm:final-graph-relations} gives the
converse at the final round.

\emph{Consequence for the comparison.} If $\RW$ identifies $D$, then
$(F,D)$ satisfies $\SG$ but not $\RW$ exactly when $\Gamma(D)$ has an
ungrounded odd component and no ungrounded bipartite one. The $15\times6$
witness of \cite[Theorem~7.5]{LQ} and the display of
Proposition~\ref{prop:full-four-distinction} are of this kind; the latter's
three odd components have sizes $5$, $13$ and $3$. In any other
signed-only display, $\RW$ fails to identify some pair.

\emph{Before any identification.} Consider a hole-free display. Before
any identification, every current class has one cell, so each vertex has a single
cross pair and at most one rectangle. The only pairs either rule can then
resolve are degenerate and complementary pairs, and the two closures agree.

\end{remark}

\subsection{A signed-only full display}
\Needspace{9\baselineskip}
\begin{proposition}[Full occupancy does not make the criteria agree]
\label{prop:full-four-distinction}
There is a full-support $16\times4$ display of length $43=U(16,4)$
on a maximum singleton base which satisfies $\SG$ but not $\RW$.
\end{proposition}
\begin{proof}
Number cells by $a=4r+c$, with indices starting at zero. The singleton
row supports are (writing $01$ for $\{0,1\}$, and similarly)
\[
 01,02,03,12,13,23,2,2,3,3,3,3,3,3,3,3.
\]
Take the following unordered pairs:
\begin{align*}
 &(2,3),(5,60),(7,49),(9,62),(10,57),(12,21),(15,29),\\
 &(16,50),(18,61),(24,31),(25,48),(27,28),(32,38),(33,40),\\
 &(34,36),(37,46),(41,56),(44,45),(20,54),(42,52),(53,58).
\end{align*}
The 22 singletons form a maximum base by Lemma~\ref{lem:cap}.
Every other cell occurs exactly once, so the length is $43=U(16,4)$.

Here the distinction is entirely in orthogonality: every selected pair
normalizes recursively. Packet F16 gives a 31-step normalization
sequence and a checker. After these identifications, form the
graph on the $\binom{43}{2}=903$ products of distinct selected classes.
Line products are grounded; rectangle equations reverse signs along edges.
A direct inventory gives exactly three components without a zero ground,
of sizes $5,13,3$. Each contains a triangle. For example, writing
$A=[53,58]$, $B=[42,52]$, $C=[41,56]$ for three paired classes
and $[a]$ for a singleton class, one triangle has vertices
\[
 [0]A,\qquad [1]B,\qquad [2,3]C.
\]
Its rectangle witnesses are $(0,13;0,1)$, $(0,10;1,2)$ and
$(0,14;0,2)$, where $(r,s;c,d)$ denotes rows $r,s$ and columns $c,d$. Each product denotes the inner product of the corresponding
class vectors. The signed rule grounds all three odd components. The
recursive rules cannot enter them. The F16 reader reconstructs the entire
product graph and checks all three triangles, proving the two outcomes.
\end{proof}

This example does not separate the optimized values, even on this base.
Replacing $(20,54),(42,52),(53,58)$ by
$(20,42),(52,58),(53,54)$ recovers the recursive length-$43$
construction of \cite[Figure~2]{CC}. Thus full occupancy does not justify keeping the pairing unchanged
when converting a signed display to a recursive one.

\section{Finite SOS upper bounds}
\label{sec:finite}

\begin{theorem}\label{thm:small}
For these three rectangles, $\zr=\zs=\zt$, with
\[
 \zt(5,6)=20,\qquad \zt(6,6)=24,\qquad \zt(7,6)=28.
\]
\end{theorem}

We describe the three logically separate parts of the proof: coverage
of the mathematical domain, validity of every reduction, and the final
propositional contradiction. Full coordinate and clause data are kept in
the companion files indexed in Appendix~\ref{app:data}.

\subsection{Maximum bases and lower witnesses}

Pair counting and $\binom d2\ge2d-3$ give ordinary maxima $14,16,18$.
The first is most easily classified after transposition to $6\times5$.
Equality forces two degree-three rows and four degree-two rows. The
triples meet once, say $\{a,b,c\}$ and $\{a,d,e\}$; the remaining rows
are exactly $bd,be,cd,ce$. There are $5\cdot3=15$ row-unordered bases
on labeled columns, all isomorphic.

At $6\times6$, the only pair-count-compatible profiles at total 16 are
$3^4 2^2$, $4\,3^2 2^3$, and $3^5 1$. The last is impossible because
each column lies in at most two triple rows. In the middle profile,
both triples must contain the two columns outside the degree-four row,
creating a $C_4$. For $3^4 2^2$, every column lies in exactly two of the
four triples. The columns identify with the six edges of $K_4$ and the
triple rows with its stars. The two pair rows choose two of its three
opposite-edge pairs. This gives $30\cdot3=90$ row-unordered labeled
bases and one isomorphism class.

At $7\times6$, equality forces $3^4 2^3$. The same four triple rows
occur, and all three opposite-edge pairs are required. There are 30
row-unordered labeled bases, again one isomorphism class. Exhaustive row-mask
enumeration and all column permutations independently check these
classifications and provide maps to the three canonical inputs.

The lower displays have $(|F|,|D|,H)=(14,6,4),(16,8,4),(18,10,4)$.
The last two are the published Figure~2 witnesses of \cite{CCQ}; the
first is an explicit additional construction. Actual-cell certificates
use respectively 227, 345 and 468 inferences and certify all 190, 276
and 378 pairs of selected classes. Thus their displayed lengths are
20, 24 and 28 by soundness.

The cell caps are 22, 26 and 30. By complete-square heredity, excluding
irreducible totals 21, 25 and 29 also excludes the respective next total:
delete one complete pair square from a hypothetical irreducible display.
The maximum singleton base remains unchanged.

\subsection{Exact shortening certificates}

The input data enumerate every unary choice and every disjoint binary
choice not already excluded. Theorem~\ref{thm:unary} explains why unary
filtering on a maximum base can be exact. The frozen exclusions additionally
carry explicit polynomial identities, rather than relying on a Boolean
criterion label. Retained choices need not all be irreducible; completeness
requires that no irreducible choice be removed without proof.

We use explicit identities and coefficient-preserving rank reductions. First, an explicit SOS with fewer
squares is checked by expanding every quartic coefficient over $\Z$ or
$\Q(\sqrt3)$. Second, a nonzero zero-diagonal symmetric matrix $H$ with
$f^THf=0$, where $f$ is the vector of displayed forms, gives a rank
reduction. Since $\operatorname{tr}H=0$, its largest eigenvalue
$\lambda$ is positive, and $I-H/\lambda$ is singular PSD. Thus
$f^T(I-H/\lambda)f=f^Tf$ has fewer squares. This uses an exact
coefficient identity, not a numerical eigenvalue decision.

Fourteen further reducible subdisplays have explicit positive-rational-weight
identities of the form
\begin{equation}\label{eq:weightedidentity}
 \sum_{j=1}^{N}f_j^2=\sum_{i=1}^{N-1}\alpha_i\ell_i^2,
 \qquad \alpha_i\in\Q_{>0},\quad
 \ell_i\in\Q[x_1,\ldots,x_m,y_1,\ldots,y_n]_{(1,1)}.
\end{equation}
Here the subscript denotes bilinear forms. Clearing denominators reduces
verification to integer coefficient expansion; each
$\sqrt{\alpha_i}\ell_i$ is a real bilinear form. There is one additional
core at $5\times6$, three at $6\times6$ and ten at $7\times6$.
Each identity names the complete original singleton and pair supports,
so a source hash and a checked coordinate map identify exactly which
reducible subfamily is excluded. The accompanying explicit-core manifest
binds all fourteen identities to these original sources.

\begin{samepage}
\paragraph{A literal ten-to-nine reduction.}
For the first additional $6\times6$ core, the reduction also has a simple
explicit form. Write $w_j=x_{\lfloor j/6\rfloor}y_{j\bmod6}$ with
zero-based indices, put $p=w_{12}+w_{35}$, and let
\[
 L=\sum_{j\in\{0,2,17,20,23,24,29\}}w_j^2
   +(w_5+w_{32})^2+p^2+(w_{14}+w_{18})^2.
\]
The rectangle identities give
\begin{align}\label{eq:explicitcore}
 2L={}&(w_0-w_{14}+w_{23})^2+(w_0-w_{29})^2
       +(w_5-w_{14}-w_{18})^2\notag\\
 &+(w_5+w_{24}+w_{32})^2+(w_{29}-w_{18})^2
       +(w_{32}-w_{23}-w_{24})^2\notag\\
 &+(w_2-w_{20})^2+(w_2+p+w_{17}+w_{20})^2+(w_{17}-p)^2.
\end{align}
Dividing these nine forms by $\sqrt2$ shortens the ten-square display
$L$. Restoring the nine singleton squares at
$1,6,9,10,13,15,22,31,34$ gives an explicit $19$-to-$18$ reduction
of the same core. Its verification needs only integer coefficient
expansion. 
\end{samepage}

Restoring the omitted displayed squares transports a shortening of a
complete subfamily to every completion containing it. Row/column
substitutions are checked on the full polynomial. A local square
$6\times6$ identity may also be transposed before embedding in six
rows of a $7\times6$ rectangle; this is not a transposition of the
whole nonsquare board.

\subsection{Guarded row deletion and propositional coverage}

The $7\times6$ canonical base has three degree-two rows. Deleting one
leaves a maximum $6\times6$ base with 16 singleton cells; deleting two
leaves a maximum $5\times6$ base with 14. Let $t$ count distinct selected
pairs touching the deleted union. From a hypothetical irreducible
29-square display, the retained lengths are $27-t$ or $25-t$.
The already proved smaller cases give
\begin{equation}\label{eq:deletioncuts}
 t\ge3\quad\text{or}\quad t\ge5,
 \qquad\text{respectively}.
\end{equation}
A pair meeting both deleted rows is counted once. If $H$ is the number
of holes in the union and $I$ the number of pairs wholly inside it,
cell coverage gives $t=4-H-I$ or $t=8-H-I$. Therefore the respective
constraints are $H+I\le1$ and $H+I\le3$. Every retained singleton base
and its embedding is checked; deleting arbitrary degree-three rows
would not justify the same limited bounds.

For each target, a Boolean variable chooses each allowed pair and one
variable marks each possible hole. Pair conflicts enforce disjointness;
cell clauses enforce either a hole or an incident selected pair, and
pair/hole conflicts exclude both. All negative triples of hole variables,
together with a positive clause omitting each hole variable in turn,
express exactly two holes. Hence the selected pair count is forced to
be seven, nine or eleven. A verified reducible subfamily contributes
the clause that not all its pairs may be chosen.

\begin{center}
\begin{tabular}{rrrrrr}
\toprule
Rectangle&Target&Pair variables&Hole variables&Clauses&\shortstack{Additional\\explicit cores}\\
\midrule
$5\times6$&21&58&16&1,416&1\\
$6\times6$&25&90&20&3,070&3\\
$7\times6$&29&138&24&8,395&10\\
\bottomrule
\end{tabular}
\end{center}
There are no auxiliary variables or symmetry restrictions on the
matching search. Symmetries are used only to transport valid reductions.
The $7\times6$ total comprises 4,526 matching clauses, 609 guarded
deletion clauses and 3,260 reduction images.\footnote{These are 8,395
clause lines and 8,383 distinct clauses. Twelve hole/pair conflicts are
repeated between the matching and deletion lists; repetitions do not
affect satisfiability.} The mathematical input
check reconstructs every clause. Separate DRAT certificates are then
checked against the exact final formulas \cite{DRAT}. Thus every
irreducible target would give a satisfying assignment, whereas each
formula is unsatisfiable. This proves the required three exclusions and,
with the base and lower certificates, Theorem~\ref{thm:small}.

\subsection{The seven-by-seven square}

\begin{theorem}\label{thm:seven-square}
$\zr(7,7)=\zs(7,7)=\zt(7,7)=33$.
\end{theorem}
\begin{proof}
Pair counting and $\binom d2\ge2d-3$ give $|F|\le21$.
Equality forces seven degree-three rows covering every column pair once.
Up to column relabeling, the triples through column 0 are
$012,034,056$. The remaining triples through column 1 may then be taken as
$135,146$, forcing the remaining triples $236,245$.
Thus the Fano base is the unique maximum base up to row and column
permutations. A 21-singleton, 12-pair witness on this base has an
actual-cell recursive certificate, giving the lower value 33.

A target of 34 squares has 13 pairs and two holes among the 28 free
cells. Of the 378 possible individual pairs, the 210 omitted choices
have verified SOS shortenings. The complete matching formula represents
all remaining choices with 168 pair variables and 28 hole variables. The additional cuts use
exact polynomial reductions and verified row/column substitutions.
There is also an acyclic use of $\zt(6,6)=24$: deleting the row and
column of any fixed incidence leaves a maximum 16-cell base. Retain
exactly the displayed squares whose complete supports avoid that row
and column. This sub-sum is irreducible by heredity. If $T$ distinct
selected pairs meet the deleted region, its length is $29-T$, so
$T\ge5$. Only these incidence deletions are used;
a nonincident deletion would leave a smaller singleton base and would
not justify the same bound.

The complete 196-variable, 15,666-clause formula is reconstructed from
these conditions, without symmetry restrictions on matchings, and has
a separately checked DRAT refutation.
The ordinary classification transports this exclusion to every maximum
base. The cell cap is 35, and complete-square heredity excludes a
hypothetical irreducible 35-square display by deleting one pair square.
Hence $\zt(7,7)\le33$, completing the proof. The source coordinates,
the reductions, input audit and proof check are indexed in packet F7.
\end{proof}

\subsection{Further four- and five-column upper bounds}
\label{sec:further-finite}

\begin{theorem}\label{thm:further-finite}
The three parameters have the following common values:
\[
\begin{array}{c|ccccc}
(m,n)&(7,5)&(8,5)&(10,5)&(9,4)&(10,4)\\
\hline
\zr(m,n)=\zs(m,n)=\zt(m,n)&24&27&34&24&27.
\end{array}
\]
\end{theorem}

\begin{proof}
Since $\zr\le\zs\le\zt$, it suffices to give recursive-line lower
displays and upper bounds for $\zt$. The actual upper
bounds use the complete-matching method of this section: each excluded
local pattern has a shorter real SOS, while inherited bounds are applied
only to complete selected squares on a maximum base of the smaller
rectangle. No rejected recursive or signed certificate is an upper-bound
premise. Each target exceeds the asserted value by one:
\begin{center}
\begin{tabular}{ccrrrr}
\toprule
Rectangle&Maximum bases&$|F|$&Target pairs&Target holes&Target length\\
\midrule
$7\times5$&4&15&10&0&25\\
$8\times5$&1&17&11&1&28\\
$10\times5$&1&20&15&0&35\\
$9\times4$&3&15&10&1&25\\
$10\times4$&5&16&12&0&28\\
\bottomrule
\end{tabular}
\end{center}

Here ``maximum bases'' counts row/column isomorphism classes. Their
coverage is elementary. On five columns, pair counting gives the upper
bounds $15,17,20$. At seven rows, either two triple rows use all three
pairs in each triple, leaving four pair rows and one singleton row, or
one triple is accompanied by six of the seven remaining column pairs.
The triples in the first case meet in one column; the singleton is
anchored either there or in an outer column. In the second case the
omitted pair lies either outside the triple or crosses it. These are
four types, with 145 row multisets on labeled columns in total. At eight
rows the base is one triple and all seven remaining pair rows; at ten
rows it is all ten pair rows. Larger row degrees are excluded by the
same pair budget and $C_4$-freeness.

For four columns every maximum base consists of the six pair rows and
singleton rows. At nine rows the anchor multiplicities are
$(3,0,0,0)$, $(2,1,0,0)$ and $(1,1,1,0)$. At ten rows they are
$(4,0,0,0)$, $(3,1,0,0)$, $(2,2,0,0)$, $(2,1,1,0)$ and $(1,1,1,1)$. Thus the table covers every maximum base.

For each type, every pairing with the stated hole count either appears
in the matching domain or is excluded by a verified local reduction. The local exclusions are checked by
exact coefficient identities or certified lower-rank real Gram matrices.
The identities use integers, positive rational weights or $\Q(\sqrt3)$.
For the real Gram matrices, exact rational bounds keep the principal
block positive definite and make an invertible-preconditioned Schur map
a strict contraction of a closed box into itself. Its fixed point has
zero Schur complement, giving a positive-semidefinite matrix of the
smaller principal rank. The checks use exact LDL decompositions and
rational arithmetic; rationality of the final matrix entries is not assumed. The maximum-retaining seven-row, five-column deletion clauses use
$\zt(6,5)=20$ (Theorem~\ref{thm:small}, transposed).
The four-column inputs also use $\zt(6,4)=16$ from
\cite[Corollary~4.23]{LQ} and $\zt(7,4)=19$, $\zt(8,4)=21$ from
\cite[Theorems~4.2 and~4.4]{XY}; the ten-row input uses the nine-row
bound only after that separate proof is complete. Only the $10\times5$ formula uses lexicographic constraints: all $120$
column permutations, with their induced permutations of the pair rows,
act on its matching variables. A hypothetical irreducible display and
every image under this group satisfy every matching clause and every
reduction clause: each reduction is certified independently of the lex
constraints. The lexicographically least image therefore satisfies both the
reduction clauses and the inclusive lex comparisons; the reduction-clause
list need not itself be invariant.
Independent input reconstruction and
complete clausal refutations exclude every target in the table. The
lower certificates and corresponding exact upper-bound packets are
indexed in Appendix~\ref{app:data}. The four-column recursive values
were already known from~\cite{CC}; here the new conclusion is their
actual-SOS optimality.
\end{proof}

The value $\zt(10,5)=34$ also answers the five-column case of
\cite[Section~8, Open Problem~1]{CLQIncidence}.

\section{A uniform six-column construction}\label{sec:sixconstruction}

The remaining part of Theorem~\ref{thm:six} is a construction for every
$m\ge15$. All coordinates in the following recipes are zero-based.
Start from the $15\times6$ witness in \cite[Figure~4]{CCQ}, preserving
its 30 singleton cells. Replace four of its selected pairs by chains:
\begin{center}
\begin{tabular}{ccccc}
\toprule
Track&Left cell&Right cell&Outgoing/incoming columns&Singleton column\\
\midrule
0&$(10,3)$&$(12,4)$&$(3,4)$&5\\
1&$(0,2)$&$(8,5)$&$(2,5)$&1\\
2&$(4,0)$&$(13,2)$&$(0,2)$&3\\
3&$(10,1)$&$(11,4)$&$(1,4)$&0\\
\bottomrule
\end{tabular}
\end{center}
Write these columns as $(a_i,b_i)$ and the new row on track $i$ in
block $k$ as $u_i(k)=15+4(k-1)+i$. With $u_i(0)$ and $u_i(t+1)$
the actual endpoint rows, select
\begin{equation}\label{eq:sixchain}
 Q_i(k)=\{(u_i(k),a_i),(u_i(k+1),b_i)\},\qquad0\le k\le t.
\end{equation}
In each block, place its four singletons at the listed columns and
same-column bridges between track pairs $01,02,03,12,13,23$ in
columns $0,1,2,4,3,5$, respectively. The normalized block is
\begin{equation}\label{eq:sixblock}
\begin{array}{c|cccccc}
 &0&1&2&3&4&5\\\hline
0&B_{01}&B_{02}&B_{03}&Q_0(k)&Q_0(k-1)&S_0\\
1&B_{01}&S_1&Q_1(k)&B_{13}&B_{12}&Q_1(k-1)\\
2&Q_2(k)&B_{02}&Q_2(k-1)&S_2&B_{12}&B_{23}\\
3&S_3&Q_3(k)&B_{03}&B_{13}&Q_3(k-1)&B_{23}
\end{array}
\end{equation}
where bridge and singleton indices $k$ are suppressed. Every cell is
occupied exactly once, and every added singleton row has degree one.

\subsection{Uniform certificate principles}\label{sec:certificate-principles}

We make explicit what permits a finite certificate to establish an
arbitrary-length family. A chain whose consecutive pairs use fixed
outgoing and incoming columns has the following wave property: once
one selected pair is identified, its representative in the outgoing
column grounds the companion of the next selected diagonal. One
rectangle identifies that diagonal. Reverse propagation uses the
incoming column. At either end the argument uses the actual core cell,
not a fictitious extra module row.

A finite collar consists of the core and a fixed number of first and
last blocks. A cell paired with one outside the collar retains its
full-family class and is not treated as a hole.

\begin{lemma}[Lifting symbolic certificates]\label{lem:certificate-lifting}
Consider a family of finite displays indexed by an integer $t\ge t_0$.
A finite list of record templates certifies the family if the following
conditions hold on every template's entire integer domain:
\begin{enumerate}[label=(\roman*)]
\item Each record specifies cells and classes by affine integer
expressions in $t$ and local indices. Cell expressions evaluate to cells of the display, and class indices
refer to its prescribed classes. Occupancy, membership, coincidences and
distinctions hold throughout the domain. Every rectangle used has distinct rows and columns.
\item Each instantiated record is a valid line, hole, rectangle,
complementary-pair or saturation inference from Section~\ref{sec:prelim},
or an application of an already
proved lemma. Each noninitial premise is the conclusion of an earlier instantiated
record in an explicit well-founded order.
\item The conclusions cover all selected-pair identifications and
all orthogonality relations between distinct final classes.
\end{enumerate}
Then every display in the family satisfies $\RW$.
\end{lemma}
\begin{proof}
Fix $t\ge t_0$. Condition (i) makes every instantiated record a statement
about this finite display. Well-founded induction using (ii) proves each
record from the coefficient rules or the cited lemma. Condition (iii)
gives the terminal conditions of $\RW$. Since $t$ was arbitrary, the conclusion holds throughout the family.
\end{proof}

In the certificates below, the domain checks are exact integer equalities,
inequalities and case splits. The orders are given by successive
normalization stages, decreasing distances between indices, or distances
to a verified boundary. For mixed products, a graph vertex specifies a
core class and a module type; an edge specifies a genuine rectangle and
an integer shift of the module index. A path terminates at a known ground
or an already certified boundary class. A closed walk with nonzero total
shift can be iterated only after checking that every exit reaches such a
boundary. The written recurrences below handle the exceptional
components; the remaining record tables and domain checks are in the
companion. This is a symbolic computer-assisted proof, not an inference
from a finite sample of lengths.

\subsection{Normalization and all-index orthogonality}

For the $15+4t$ residue, four core/end-block certificates initiate the
chains in order $1,2,3,0$, using 44, 28, 13 and 19 primitive facts.
The wave argument completes each track before it is used in the next
stage. A further 118-fact certificate identifies the 26 retained core
pairs. Bridges identify by the column rule. These windows are constant
as coefficient systems for all $t\ge3$.

After normalization, seven short symbolic rectangle words establish
the boundary-chain families
\[
 Q_1(0)\perp Q_0(k),\quad Q_0(t)\perp Q_2(k)
 \quad(0\le k\le t),\qquad
 Q_1(0)\perp Q_3(t),\quad Q_2(t)\perp Q_3(t).
\]
They terminate at actual line grounds and use no core orthogonality
conclusion. The module certificate then combines inward index
transfers, near-diagonal windows, boundary/interior paths and a
two-end window. Its verified domains cover every pair of module
classes for every $t\ge17$. The inward steps strictly decrease the
distance between block indices; the bounded-distance and endpoint
cases are separate finite templates.

There are 56 old nonport classes and 14 module phase types. Of the
784 mixed states, 782 have paths of at most ten transfers, with block
shifts in $[-2,3]$, to a line or a previously established module ground.
These paths apply for $3\le k\le t-3$. The two exceptions have a
common recurrence. Write $F=v_{1,0}$ for the singleton coefficient vector
and $E=v_{1,1}=v_{1,4}$ for the normalized coefficient vector of the
retained core pair. With
\[
 X_k=\langle F,Q_3(k)\rangle,\qquad
 Y_k=\langle E,S_3(k)\rangle,
\]
a fixed three-rectangle word gives $X_0=0$, and the block gives
\begin{equation}\label{eq:sixexception}
 X_{k-1}=-Y_k=X_k\qquad(1\le k\le t).
\end{equation}
Hence both entire exceptional families vanish. A core plus first-five
and last-five block patch covers the remaining old/old and boundary
targets. Its 6,282 facts cover 5,684 required targets; its affine
geometry is valid for $t\ge11$. For $t\ge17$ the interior and boundary
ranges cover every index. Irrelevant cut-pair variables are not assumed
zero. Every one of the 17 shorter lengths is checked separately.

\subsection{The other three residues}

For the $17+4t$ family, add two core rows with singleton column 0 and
the five pairs
\[
 (15,1)(15,2),\quad(15,3)(15,4),\quad(16,1)(16,5),\quad
 (15,5)(16,3),\quad(16,2)(16,4).
\]
The module now starts after row 16, with the same original four ports.
No added pair crosses to the old core. The two-row subsystem has its
own seven-form certificate. The enlarged mixed graph has 882 states,
880 ordinary paths and exactly the same two exceptions
\eqref{eq:sixexception}. Its core/boundary patch and the same index
coverage complete this residue.

To either odd residue, attach one row with singleton in column 1, a
hole in column 0, and pairs in columns $(2,5)$ and $(3,4)$.
The finite boundary certificates move this last row to core position
15 or 17 and shift the module rows accordingly. This row permutation
preserves every selected support; it is not an additional attachment.
For the $16+4t$ family, 40 intrinsic phase paths ground the ordinary
new/old products. If the new singleton is $F'$ and its $(2,5)$ pair
is $A'$, the two exceptions satisfy
\[
 \langle F',Q_1(k-1)\rangle
   =-\langle A',S_1(k)\rangle
   =\langle F',Q_1(k)\rangle.
\]
The actual right endpoint supplies a seed through the retained pair
with cells $(8,1),(9,2)$, which shares column 2 with $A'$.
A boundary patch uses this seed and only already proved facts of the
$15+4t$ family, so the argument is not circular.

For $18+4t$, the $17+4t$ subsystem, the $16+4t$ subsystem, and the fixed
three-row extra subsystem jointly cover every selected class pair.
No selected support is cut, so their certificates embed. The four
residue counts are
\begin{center}
\begin{tabular}{rrrrr}
\toprule
$m$&$|F|$&$|D|$&$H$&$N$\\
\midrule
$15+4t$&$30+4t$&$30+10t$&0&$60+14t$\\
$16+4t$&$31+4t$&$32+10t$&1&$63+14t$\\
$17+4t$&$32+4t$&$35+10t$&0&$67+14t$\\
$18+4t$&$33+4t$&$37+10t$&1&$70+14t$\\
\bottomrule
\end{tabular}
\end{center}
The uniform proofs apply for $t\ge17$, and the 68 shorter configurations
are covered by finite certificates. In every case $|F|=m+15$ and the
display attains $U_6(m)$. This proves the infinite part of
Theorem~\ref{thm:six}.

For $m=2,3,4$, transposition and the prior two-column identity,
three-column theorem and $6\times4$ equality give $7,12,16$
\cite{LQ}. Theorem~\ref{thm:small} supplies $m=5,6,7$.
For $8\le m\le14$, \cite[Corollary~4.2]{CCQ} gives the actual cap,
which, by Lemma~\ref{lem:ordinarysix}, simplifies to
$\fl{(15m+15)/4}$. All parts of
Theorem~\ref{thm:six} now follow.

The least eventual threshold for equality to the cell cap $U(m,6)$ is
$8$: at $m=7$ the common value is $28$ and the cap is $30$.
The formula $\lfloor(7m+15)/2\rfloor$ starts at $m=14$, where the value
is $56$; at $m=13$ it gives $53$ instead of $52$.

\section{The odd-width families}\label{sec:otherwidths}

The width-five formula in Theorem~\ref{thm:otherwidths} follows from
\cite[Theorem~5.4 and Corollary~5.5]{XiChangFive} and the parameter hierarchy.
The remaining three families share a one-row module. Their finite
interfaces are interpreted using the certificate principles of
Section~\ref{sec:certificate-principles}.

A related distance-reducing rectangle argument is given in
\cite[Lemma~4.2]{XiChangFive}.

\begin{lemma}[Parallel one-row chains]\label{lem:parallel}
In real module rows $k=1,\ldots,t$, $t\ge1$, place $Q_i(k+1)$ in
column $a_i$ and $Q_i(k)$ in column $b_i$, for disjoint column pairs
$(a_i,b_i)$. An optional singleton $S(k)$ occupies one further column.
The first and last classes have their other representatives in the core,
in columns $a_i$ and $b_i$, respectively.
After the target tracks normalize, all distinct module classes are
recursively orthogonal.
\end{lemma}
\begin{proof}
Classes on the same track share a column. For distinct tracks, put
$U(q,p)=\langle Q_i(q),Q_j(p)\rangle$. Equal or adjacent indices
have a common module row; the case $q=p=t+1$ uses outgoing
representatives in row $t$. If $q>p+1$, an outgoing/incoming rectangle
gives
\begin{equation}\label{eq:paralleldrift}
 U(q,p)\longleftrightarrow U(p+1,q-1),
\end{equation}
reducing the absolute gap by two. The symmetric move handles $p>q+1$.
The larger index uses its outgoing representative first, so every row
lies in $1,\ldots,t$, even at the upper endpoint. For $Q_i(q)S(p)$, the row grounds are $q=p,p+1$; an incoming swap
handles $q<p$, followed where needed by the same gap reduction.
Distinct singleton classes share a column. Only tracks appearing in
the target are needed, so the lemma is available during staged
normalization without presuming other tracks.
\end{proof}

\subsection{Widths seven and nine}

For widths seven and nine, the cores and new row patterns are
as follows. An entry $(a_i,b_i)$ specifies the outgoing/incoming
columns, and $s$ the singleton column.
\begin{center}
\begin{tabular}{cccccc}
\toprule
$n$&Core rows&Core $(|F|,|D|,H)$&$s$&Track columns&$N$ after $t$ rows\\
\midrule
7&25&$(46,64,1)$&0&$(1,2),(3,4),(5,6)$&$110+4t$\\
9&36&$(72,126,0)$&0&$(3,5),(2,4),(6,1),(8,7)$&$198+5t$\\
\bottomrule
\end{tabular}
\end{center}
The full core lists and port indices are in the corresponding data
packets. Replacing each port by its chain preserves every other pair.
Each new row is full and has one singleton, so the bases attain
$m+\binom n2$ and their hole counts stay $1,0$. The resulting
lengths are the width-seven and width-nine rows of
Theorem~\ref{thm:otherwidths}.

For $t\ge5$, retain the core and the first two and last two module
rows. The cut classes $Q_i(3)$ and $Q_i(t-1)$ are distinct and occupied.
Staged finite certificates, with Lemma~\ref{lem:parallel} used only
after its target tracks normalize, identify the whole core and all
chains. The remaining mixed graphs have the following complete counts.
\begin{center}
\begin{tabular}{rrrr}
\toprule
$n$&Core-represented classes&Mixed vertices&Components\\
\midrule
7&113&452&20\\
9&202&1,010&30\\
\bottomrule
\end{tabular}
\end{center}
Splitting a port gives its two cells distinct classes, both with core
representatives. At both widths, every component has a grounded spanning tree. The width-nine collar
uses 64 staged transports, each on two consecutive real rows with
verified net shift one and only previously normalized target tracks.
The resulting repetitions are valid throughout the module.

Each lifted mixed path either reaches its declared ground or reaches
an actual first/last chain class, where the established core clique
or the normalized diagonal prescription applies. The affine collar checks exclude accidental coincidences of omitted
mates over the full domain $t\ge5$, including its boundary $t=5$.
At each width, the five cases $0\le t<5$ are separately certified.

At width seven, separate actual-cell certificates cover rows 19 through
24. Their maximum singleton counts are 39, 40, 42, 43, 44 and 45, and
their displayed lengths are $4m+10$. Ordinary optimality follows from
$2|F|-3m\le21$ at row 19, the strict subterminal bound of
Lemma~\ref{lem:cap} at row 20, and the terminal formula thereafter.
Together with the uniform family from row 25, this proves the claimed
range. At row 18 the degree bound gives $|F|\le37$ and hence
$\zt(18,7)\le81<82$, proving the least linear threshold 19.
At width nine the new 36-row core starts the uniform family directly.
At row 35, Lemma~\ref{lem:cap} gives $|F|\le70$ and therefore
$\zt(35,9)\le192<193$. Its least linear threshold is consequently 36.

\subsection{Eleven columns}
The direct $55\times11$ core has 110 singletons, 247 pairs and the
single hole $(19,9)$. The outgoing/incoming column pairs of
$Q_0,\ldots,Q_4$ are $(1,6),(8,3),(7,9),(5,10),(4,2)$, with
singleton column 0. Adding $t$ one-row modules gives
$|F|=110+t$, $|D|=247+5t$ and $N=357+6t=6m+27$ at $m=55+t$.

The staged collar preserves the genuine hole and all occupied cut
mates. The staged proof uses 66 unit-shift transports and establishes
normalization and the 362-class source clique before mixed closure. Twenty-six mixed
components have direct grounds. The remaining two use the normalized
coefficient vectors
\[
\begin{aligned}
 F_3&=v_{6,0},& D_3&=v_{6,5}=v_{6,10},\\
 F_4&=v_{7,0},& D_4&=v_{7,2}=v_{7,4}.
\end{aligned}
\]
Here each $D_j$ is the class of an actual same-row selected pair.
Rectangles in core row 6 or 7 and the real module row $54+k$ give
\begin{equation}\label{eq:elevenrecurrence}
 \langle F_j,Q_j(k)\rangle
   =-\langle D_j,S(k)\rangle
   =\langle F_j,Q_j(k+1)\rangle,
 \qquad j=3,4,\quad1\le k\le t.
\end{equation}
The initial classes are $Q_3(1)=v_{11,5}$ and $Q_4(1)=v_{43,4}$.
They are distinct from $F_3,F_4$, respectively, so the previously
proved source clique supplies both zero seeds. Induction exhausts
these two components, including the terminal chain index $t+1$.
These recurrences handle the two exceptional components. Exact affine checks cover
all $t\ge5$, and the five shorter cases are separately certified.
The maximum base and cell cap prove equality for every $m\ge55$.
Two additional full displays at $m=53,54$ have
$(|F|,|D|,H)=(107,238,0),(108,243,0)$ and lengths $345,351$.
Their bases consist of one triple row and the 52 pair rows not covered
by that triple, with one extra singleton row at $m=54$. Their pair
budgets give $2|F|-3m\le55$, proving ordinary optimality. Explicit
coefficient certificates establish $\RW$ for both displays and bridge
the range to the uniform family from $m=55$.
At $m=52$, the same bound gives $|F|\le105$ and hence
$\zt(52,11)\le338<339=6\cdot52+27$. Thus 53 is the least threshold.
Packets O11 and O11\_BRIDGE contain the uniform proof and the two finite
certificates.

\section{Signed exclusions and small-width equality}\label{sec:signed-small}
\subsection{Four columns}
\begin{proposition}\label{prop:four-signed}
One has $\zr(11,4)=\zs(11,4)=29$ and $\zr(12,4)=\zs(12,4)=32$.
\end{proposition}
\begin{proof}
For $m\ge6$, pair counting gives $z(m,4)=m+6$. Equality forces six
rows containing the six distinct column pairs and $m-6$ singleton rows.
The nonincreasing quadruples of anchor multiplicities classify all maximum
bases up to row and column permutations. There are six types at eleven
rows and nine at twelve rows. The respective cell caps are 30 and 33.
Complete signed enumeration excludes both targets:
\begin{center}
\begin{tabular}{lrrr}
\toprule
Grid&Maximum-base types&Target pairs&Search states\\
\midrule
$11\times4$&6&13&189,363,352\\
$12\times4$&9&15&379,946,285\\
\bottomrule
\end{tabular}
\end{center}
Every base is exhausted, with one hole in the eleven-row target and none
in the twelve-row target. The per-base inputs and counts are recorded in S4.

The search branches on a remaining
cell: either it is a hole if the remaining budget permits, or it is paired
with each possible partner. Failed signed subfamilies are excluded by complete-pair heredity.
Deleting a complete displayed form grounds all rectangle deviations incident
with its removed cells, so the signed derivation on retained cells replays.
A matching-component bound rejects only insufficient remaining capacity.
The only symmetry reduction swaps two wholly untouched rows with identical
fixed supports, while fixing the pivot, selected pairs and chosen holes.
Thus every branch omission preserves a representative of every completion.

The independent reader matches the node and depth counts of the archived
signed union-find search. It instead
rebuilds a parity graph, colors its components by breadth-first search,
and grounds components containing a zero, an odd cycle, or opposite selected
inequalities. Only a grounded selected-pair deviation enables its
identification equations. The algorithms use the signed criterion throughout;
a recursive-line rejection is never a pruning condition. The source audit,
complete inputs, calibration cases and termination logs are in packet S4.

The recursive lower values 29 and 32 were already obtained in \cite{CC}.
Explicit twelve- and fourteen-pair witnesses are also replayed by a separate
actual-cell rule reader in S4. The hierarchy gives the asserted equalities.
These are signed upper bounds. The actual-SOS intervals remain
$29\le\zt(11,4)\le30$ and $32\le\zt(12,4)\le33$.
\end{proof}

\begin{corollary}\label{cor:four-exception}
For every $m\ge1$ other than possibly $m=14$,
$\zr(m,4)=\zs(m,4)$.
\end{corollary}
\begin{proof}
At one row every cell is fixed. The exact values for two through eight rows
are in \cite{LQ,XY};
Theorem~\ref{thm:further-finite} covers nine and ten rows, and
Proposition~\ref{prop:four-signed} covers eleven and twelve.
At thirteen rows and every $m\ge15$, the recursive constructions
in \cite[Table~1 and Theorem~1.1]{CC} attain the common cell cap. At fourteen rows, the
remaining interval is $37\le\zr\le\zs\le38$.
\end{proof}

\subsection{Signed exclusions at \texorpdfstring{$11\times5$}{11x5} and \texorpdfstring{$8\times7$}{8x7}}\label{sec:signed-exclusions}

\begin{theorem}\label{thm:signed-exclusions}
One has
\[
 \zr(11,5)=\zs(11,5)=37,\qquad \zr(8,7)=\zs(8,7)=38 .
\]
\end{theorem}

\begin{lemma}\label{lem:bases-11-5-8-7}
One has $z(11,5)=21$, and the maximum base is unique up to row and
column permutations: ten rows carry the ten column pairs once each, and
one row is a singleton.
One has $z(8,7)=22$, and there are exactly two classes of maximum bases:
\begin{itemize}
\item type~A replaces one triple of a Fano triple system by two of its
      three pairs;
\item type~B adds one singleton row to a Fano triple system.
\end{itemize}
\end{lemma}

\begin{proof}
With row degrees $d_r$, $C_4$-freeness gives
$\sum_r\binom{d_r}2\le\binom n2$.
\begin{itemize}
\item \emph{$11\times5$.} Since $11\ge\binom52$, Lemma~\ref{lem:cap} gives
  $|F|\le21$ and forces ten rows of degree two, one on each column pair,
  and one row of degree one. Column permutations are transitive on the
  singleton's column.
\item \emph{$8\times7$.} The inequality $2d\le3+\binom d2$ gives
  $2|F|\le24+21$, so $|F|\le22$. At $|F|=22$, $\sum_r(d_r-2)(d_r-3)/2\le1$,
  so every row has degree two or three, with at most one row of degree one
  or four. This leaves the degree profiles $2^2\,3^6$, $1\,3^7$ and
  $2^3\,3^4\,4$.
\end{itemize}
In the last profile, each triple meets the degree-four row in at most one
column, since the base has no $C_4$. So each of the four degree-three
rows contains at least two of the three columns outside the degree-four
row, and hence a pair of them. Two triples share at most one column, so
these four pairs are distinct. Only three such pairs exist, so this
profile is impossible.

For $1\,3^7$, the seven triples cover all $21$ column pairs, so they form
a Fano system. The remaining row is a singleton, which gives type~B.

For $2^2\,3^6$, the three column pairs missed by the six triples have even
degree at every point, so they form a triangle. Adding that triangle
completes a Fano system, and the two pair rows are two of its edges,
which gives type~A.

The Fano system is unique up to relabelling: its three triples through
one point may be labelled $012,034,056$; choosing the triple through
$1,3$ as $135$ forces the remaining triples $146,236,245$.
Its automorphism group is transitive on points and on incident
point--triple flags. Hence each type is a single class.
\end{proof}

\begin{proof}[Proof of Theorem~\ref{thm:signed-exclusions}]
\emph{Lower bounds.} Explicit recursive-line displays of lengths $37$ and
$38$ on maximum bases were checked by the reference closure and an
independent C++ closure in the recorded audit. The portable replay in
packet X115/X87 rechecks them with the reference closure.

\emph{Reduction to hole-free displays.} We have
\[
U(11,5)=\lfloor(55+21)/2\rfloor=38,\qquad U(8,7)=\lfloor(56+22)/2\rfloor=39 .
\]
In both cases $mn+z(m,n)$ is even. So by the cell count
$mn=|F|+2|D|+H$, a display of length $U$ has no hole. It therefore
suffices to show that none of the three maximum bases of
Lemma~\ref{lem:bases-11-5-8-7} carries a hole-free display satisfying
$\SG$.

\emph{Heredity.} Suppose $(F,D)$ satisfies $\SG$ and $S\subseteq D$. Then
the display $(F,S)$, in which the cells of $D\setminus S$ become holes,
also satisfies $\SG$ (Lemma~\ref{lem:heredity}). We use this in two
ways:
\begin{itemize}
\item every pair of a signed display is a \emph{candidate}, that is, a
      pair whose one-pair display is accepted;
\item a set $S$ of candidates whose display $(F,S)$ fails $\SG$ cannot lie
      in a signed display. Such an $S$ is a \emph{cut}.
\end{itemize}

\emph{Region domains.} Fix a set $R$ of rows:
\begin{itemize}
\item at $11\times5$, the four pair rows through the singleton's column;
\item on type~A, its two pair rows;
\item on type~B, the three triples through the singleton's column.
\end{itemize}
The \emph{pattern} of a display is its set of pairs with both cells in
$R$. By heredity it is itself an accepted display on candidates. An
exhaustive enumeration of all partial matchings of the free cells of $R$
lists every such pattern. Every accepted pattern ($4{,}968$, $1{,}449$
and $353$ of them) is mapped onto one of $229$, $190$ and $18$ listed
representatives by a base automorphism preserving $R$. Each map is checked
to preserve the base.

\emph{Propositional encoding.} For a base, take one variable $x_P$ for
each candidate pair $P$, and the following clauses:
\begin{itemize}
\item exactly one selected pair meets each free cell;
\item for every cut $S$ found, the clause $\bigvee_{P\in S}\neg x_P$.
\end{itemize}
Call this formula $B$. For each representative $d$, let $L_d$ be the unit
literals fixing the region pattern to $d$. A hole-free signed display,
moved by an automorphism so that its pattern is a representative,
satisfies $B\wedge L_d$ for that $d$. The formulas contain no
symmetry-breaking clauses; the identical-row swap of
Proposition~\ref{prop:four-signed} has no instance on these bases.
Hence it suffices to refute $B\wedge L_d$ for every representative $d$.
The final formula adds \emph{selector} clauses: a new variable $s_d$
for each $d$, the clauses $s_d\Rightarrow L_d$, and $\bigvee_d s_d$.

The cuts were generated by counterexample-guided search. Their
provenance is irrelevant to the proof, because every cut is re-checked.

\emph{Refutations.} Each refuted formula is a subset of $B\wedge L_d$: the unit clauses
$L_d$, the propagating clauses, and the clauses of $B$ not satisfied
by unit propagation from $L_d$. The checker verifies this subset relation line by
line.

For a few domains the refutation is split on the partner of one free cell
$c$ outside $R$. Since $B$ contains the clause $\bigvee_{P\ni c}x_P$,
refuting $B\wedge L_d\wedge x_P$ for every candidate $P\ni c$ refutes
$B\wedge L_d$.

Every refutation is a CaDiCaL LRAT proof. It is accepted by cake\_lpr, a
formally verified checker \cite{CakeLPR}, and separately by the LRAT checker of the
drat-trim distribution.

\emph{Audit.} Every clause of each final formula is classified:
\begin{itemize}
\item exactly-one clauses are rebuilt from the candidate list, and
      selectors from the representative list;
\item every other clause must be a negative clause on candidates, and its
      pair set is re-checked as a failed display, by the C++ closure and by
      the reference closure;
\item every excluded non-candidate pair is re-checked;
\item the region literals are rebuilt from the representatives.
\end{itemize}
The coverage enumeration and the candidate/non-candidate split were
repeated with the reference closure.
In addition, a random sample of $34$ domains and split cases, at least
eight for each base, was refuted again by kissat, and drat-trim accepted
all $34$ DRAT proofs.

The counts are as follows.
\begin{center}
\begin{tabular}{lrrrrl}
\toprule
Base&Candidates&Domains&Cut clauses&Split domains&Split cell cases\\
\midrule
$11\times5$&387&229&5{,}168{,}225&4&$4\times21$\\
$8\times7$ type A&326&190&946{,}334&1&16\\
$8\times7$ type B&312&18&589{,}720&1&15\\
\bottomrule
\end{tabular}
\end{center}
The cut counts refer to the final formula of each base. At $11\times5$ and
on type~A, $213$ and $189$ domains respectively were refuted from an
earlier formula of the same kind, which was audited in the same way
($3{,}956{,}309$ and $946{,}270$ cut clauses).

No hole-free signed display of length $U(m,n)$ therefore exists, so
$\zs(11,5)\le37$ and $\zs(8,7)\le38$. The hierarchy completes the proof.
\end{proof}

\begin{remark}
These are upper bounds for the signed criterion, not for the actual
square length. The actual values remain in the intervals
$37\le\zt(11,5)\le38$ and $38\le\zt(8,7)\le39$.

Together with the other finite values and the width-five formula of
Xi and Chang \cite{XiChangFive},
Theorem~\ref{thm:signed-exclusions} completes the width-five comparison
(Corollary~\ref{cor:five-and-six}). At $8\times7$
the common value $38$ lies one below the cell cap $U(8,7)=39$.

At $14\times4$ the remaining interval
$37\le\zr(14,4)\le\zs(14,4)\le38$ is the subject of
Problem~5.1 in \cite{CC}.
\end{remark}

\begin{proposition}\label{prop:nine-five}
One has $\zr(9,5)=\zs(9,5)=\zt(9,5)=31$.
\end{proposition}
\begin{proof}
Lemma~\ref{lem:cap} gives $z(9,5)\le18$. The following display has
18 singleton cells and 13 pairs, and its singleton set is $C_4$-free.
Equal numbers form one selected pair; bullets are singletons and the
circle is a hole. Packet BOUNDARY checks its recursive-line closure.
\[
\begin{array}{c|ccccc}
 &0&1&2&3&4\\\hline
0 & \bullet & \bullet & 1 & 4 & 5\\
1 & \bullet & 3 & \bullet & 7 & 4\\
2 & \bullet & 1 & 2 & \bullet & 9\\
3 & \bullet & 6 & 8 & 12 & \bullet\\
4 & 9 & \bullet & \bullet & 12 & 10\\
5 & 11 & \bullet & 13 & \bullet & 2\\
6 & 13 & \bullet & 11 & 8 & \bullet\\
7 & 10 & 5 & \bullet & \bullet & \circ\\
8 & 6 & 7 & \bullet & 3 & \bullet\\
\end{array}
\]
The base is maximum and $U(9,5)=\lfloor(45+18)/2\rfloor=31$, so the
hierarchy proves the assertion.
\end{proof}

\begin{corollary}\label{cor:five-and-six}
For every $m\ge1$, $\zr(m,5)=\zs(m,5)$. Consequently
$\zr(m,n)=\zs(m,n)$ for every rectangle with $\min(m,n)\le6$, except
possibly $14\times4$ and $4\times14$.
\end{corollary}

\begin{proof}
Both criteria are invariant under transposition. At width five, rows one
through four follow from the one-, two- and three-column values and the
$5\times4$ equality in \cite{LQ,XY}. At five rows the common value is $17$
\cite{LQ,XY}; at six rows it is $20$ by Theorem~\ref{thm:small} and
transposition. Theorem~\ref{thm:further-finite} gives $24,27,34$ at rows
seven, eight and ten. Proposition~\ref{prop:nine-five} gives $31$ at nine
rows. Eleven rows are covered by Theorem~\ref{thm:signed-exclusions},
and every $m\ge12$ by Theorem~\ref{thm:otherwidths}.
Widths one to three are in \cite{LQ}, width four is
Corollary~\ref{cor:four-exception}, and width six is Theorem~\ref{thm:six}.
This covers every rectangle with a side at most six, with the stated
exception.
\end{proof}

\subsection{A fixed-base gap}\label{sec:fixed-gap}
\begin{proposition}\label{prop:fixed-base-gap}
There is a maximum $7\times5$ base $F$ with $R(F)=8$ and $S(F)=9$,
although $\zr(7,5)=\zs(7,5)=24$.
\end{proposition}
\begin{proof}
Take the $7\times5$ base with row supports
\[
 012,\quad034,\quad13,\quad14,\quad23,\quad24,\quad1.
\]
It has 15 singletons and is ordinary maximum by pair counting.
Figure~\ref{fig:fixed-base-gap} displays an eight-pair recursive witness
and a nine-pair signed witness; packet \texttt{FIXED\_BASE\_NINE} contains their certificates.

 A complete recursive-line exclusion fixes each of the 190
possible two-hole sets and exhausts every pairing of the other 18 free
cells. Only rejection of a recursive-line subfamily prunes the search,
which is valid by complete-pair heredity. An independent physical-cell
reader exhausts 764,150 prefixes and 2,617,789 closure calls, finding
no nine-pair recursive display. Every larger recursive display would
contain such a subfamily. Thus $R(F)=8$. The signed lower witness and
the actual bound $\zt(7,5)=24$ prove $S(F)=9$.
Another maximum base at $7\times5$ admits nine recursive pairs, so
this example is consistent with the global equality in
Theorem~\ref{thm:further-finite}.

\end{proof}

\begin{figure}[ht]
\centering
\[
\begin{array}{c|ccccc}
 &0&1&2&3&4\\\hline
0&\bullet&\bullet&\bullet&6&\circ\\
1&\bullet&\circ&7&\bullet&\bullet\\
2&\circ&\bullet&\circ&\bullet&5\\
3&5&\bullet&4&1&\bullet\\
4&3&7&\bullet&\bullet&6\\
5&4&1&\bullet&2&\bullet\\
6&2&\bullet&8&8&3\\
\end{array}\qquad\begin{array}{c|ccccc}
 &0&1&2&3&4\\\hline
0&\bullet&\bullet&\bullet&7&\circ\\
1&\bullet&1&3&\bullet&\bullet\\
2&3&\bullet&4&\bullet&7\\
3&\circ&\bullet&6&5&\bullet\\
4&5&2&\bullet&\bullet&1\\
5&6&8&\bullet&8&\bullet\\
6&2&\bullet&9&9&4\\
\end{array}
\]
\caption{The same maximum singleton base at $7\times5$.
Left: eight recursive-line pairs. Right: nine signed pairs.
Each bullet is a distinct singleton, each circle is a hole, and equal
numbers within one grid form a selected pair. The right grid is signed
admissible; no recursive-line display on this base has nine pairs.}
\label{fig:fixed-base-gap}
\end{figure}

\appendix
\section{Certificates and verification scope}\label{app:data}
The manuscript gives the structural proofs and mathematical reductions
for the computer-assisted results. The infinite exact families also use
external symbolic record tables satisfying
Lemma~\ref{lem:certificate-lifting}; the finite tables and their exact
domain checks are essential supporting evidence. The component map is:
\begin{center}\small
\begin{tabular}{ll}
\toprule
Component&Claims supported\\\midrule
FINITE\_SIX&Theorem~\ref{thm:small}: actual six-column upper bounds\\
SIX\_FAMILY&All-length six-column construction\\
FINITE\_FIVE, FINITE\_FOUR&Theorem~\ref{thm:further-finite}\\
F7&The $7\times7$ equality\\
S4&Signed exclusions at $11\times4$ and $12\times4$\\
BOUNDARY&The $9\times5$ witness\\
FIXED\_BASE\_NINE&Proposition~\ref{prop:fixed-base-gap}\\
F16&Proposition~\ref{prop:full-four-distinction}\\
O5\_EXTENSION&Shared one-row certificate readers for O7\\
O7, O9&Odd-width families at seven and nine\\
O11, O11\_BRIDGE&Width eleven and its two boundary cases\\
X115/X87&Theorem~\ref{thm:signed-exclusions}\\\bottomrule
\end{tabular}
\end{center}
The finite SOS upper bounds use explicit identities and checked clausal
refutations. The S4 exclusions and the fixed-base recursive exclusion use
independent exhaustive implementations rather than separate clausal
refutations; reproducing those bounds requires rerunning the enumerations. Signed exclusions use failed signed subdisplays and
Lemma~\ref{lem:heredity}; they are not actual-SOS exclusions. X115/X87
contains the formulas, domain representatives, coverage records and
complete audit receipts described in Section~\ref{sec:signed-exclusions}.
Its readers regenerate LRAT proofs rather than storing all traces.

The supporting archives are deposited on Zenodo at
\url{https://doi.org/10.5281/zenodo.23124764}.

The main companion is \texttt{focused\_evidence.zip}, with SHA-256
\begin{center}\footnotesize\texttt{f68d42718c764764d27233ab6e0ad5d0411a2aeeaab5a3de8998c44104f6c714}\end{center}
The X115/X87 packet is
\texttt{signed\_exclusions\_evidence.zip}, with SHA-256
\begin{center}\footnotesize\texttt{bb43da8af5d939f2d476c6c74bf70730fcf1441a1c79ef6bc767659374663d53}\end{center}
The README gives per-component replay commands, dependencies and tool
requirements. Integrity checks verify bytes only; sampled replay is not
full proof regeneration. 

The general density, two-column extension and transfer-graph results have
written proofs. Auxiliary finite checks do not establish their infinite
ranges. Lean sources formalize the coefficient/graph equivalence in
Theorem~\ref{thm:final-graph-relations} and the existence of a nonzero
product relation exactly when an ungrounded bipartite component exists.
The formal statement uses occupied classes and vanishing for all real row and column vectors; the
geometric bridge is proved by finite-sum identities. The dimension count,
SOS rank consequence, closure algorithms and computational bounds remain
outside the formalization. An abstract chain-propagation induction is
also included; its geometric premises remain assumptions.
The Lean toolchain and mathlib revision are pinned.

\paragraph{Acknowledgment.}
The author thanks Liqun Qi for helpful correspondence.

\paragraph{AI assistance and responsibility.}
AI tools assisted exploratory computations, software development and
draft preparation. The author is responsible for the mathematical
claims, exposition and supporting data.

\begingroup
\setlength{\bibsep}{2pt}
\renewcommand{\bibfont}{\small}
\raggedright
\bibliographystyle{plainnat}
\bibliography{references}
\endgroup
\end{document}